\documentclass[11pt]{article}
\usepackage{amssymb,amsmath,amsthm,fullpage,graphicx, bm, bbm, enumerate, setspace}
\usepackage[numbers]{natbib}
\usepackage[margin=1.5in]{geometry}
\usepackage{url}
\allowdisplaybreaks

\newtheorem{theorem}{Theorem}[section]

\newtheorem{lemma}{Lemma}[theorem]

\newtheorem{proposition}{Proposition}[section]

\DeclareMathOperator{\vect}{vec}
\DeclareMathOperator{\etr}{etr}

\begin{document}

\title{Random orthogonal matrices and the Cayley transform}

\author{Michael Jauch, Peter D. Hoff, and David B. Dunson \\ 
Department of Statistical Science \\ 
Duke University
}
\date{\vspace{-5ex}}
\maketitle

\begin{abstract}
Random orthogonal matrices play an important role in probability and statistics, arising in multivariate analysis, directional statistics, and models of physical systems, among other areas. Calculations involving random orthogonal matrices are complicated by their constrained support. Accordingly, we parametrize the Stiefel and Grassmann manifolds, represented as subsets of orthogonal matrices, in terms of Euclidean parameters using the Cayley transform. We derive the necessary Jacobian terms for change of variables formulas. Given a density defined on the Stiefel or Grassmann manifold, these allow us to specify the corresponding density for the Euclidean parameters, and vice versa. As an application, we describe and illustrate through examples a Markov chain Monte Carlo approach to simulating from distributions on the Stiefel and Grassmann manifolds. Finally, we establish an asymptotic independent normal approximation for the distribution of the Euclidean parameters which corresponds to the uniform distribution on the Stiefel manifold. This result contributes to the growing literature on normal approximations to the entries of random orthogonal matrices or transformations thereof.
\end{abstract}


\section{Introduction}

Random orthogonal matrices play an important role in probability and statistics. They arise, for example, in multivariate analysis, directional statistics, and models of physical systems. The set of $p\times k$ orthogonal matrices $\mathcal{V}(k,p) = \{\bm{Q} \in \mathbbm{R}^{p \times k} \mid \bm{Q}^T\bm{Q} = \bm{I}_k\},$ known as the Stiefel manifold, is a $d_{\mathcal{V}}  = pk -k(k+1)/2$ dimensional submanifold of $\mathbbm{R}^{pk}.$ There are two notable special cases: $\mathcal{V}(1,p)$ is equivalent to the unit hypersphere, while $\mathcal{V}(p,p)$ is equivalent to the orthogonal group $\mathcal{O}(p).$ Closely related to the Stiefel manifold is the Grassmann manifold $\mathcal{G}(k,p),$ the set of $k$-dimensional linear subspaces of $\mathbb{R}^{p}.$ The Grassmann manifold has dimension $d_{\mathcal{G}} = (p-k)k.$ Typically, points in the Grassmann manifold are thought of as equivalence classes of $\mathcal{V}(k,p),$ where two orthogonal matrices belong to the same class if they share the same column space or, equivalently, if one matrix can be obtained from the other through right multiplication by an element of $\mathcal{O}(k).$ In Section \ref{Cayley_Grassmann}, we elaborate and expand upon the contributions of \citet{Shepard2015} to provide another representation of the Grassmann manifold $\mathcal{G}(k,p)$ as the subset $\mathcal{V}^{+}(k,p)$ of $p \times k$ orthogonal matrices having a symmetric positive definite (SPD) top block. In this article, we focus on orthogonal matrices having fewer columns than rows.

Both the Stiefel and Grassmann manifolds can be equipped with a uniform probability measure, also know as an invariant or Haar measure. The uniform distribution $P_{\mathcal{V}(k,p)}$ on $\mathcal{V}(k,p)$ is characterized by its invariance to left and right multiplication by orthogonal matrices: If $\bm{Q} \sim P_{\mathcal{V}(k,p)},$ then $\bm{U}\bm{Q}\bm{V} \stackrel{\text{dist.}}{=} \bm{Q}$ for all $\bm{U} \in \mathcal{O}(p)$ and $\bm{V} \in \mathcal{O}(k).$ Letting $l: \mathcal{V}(k,p) \to \mathcal{G}(k,p)$ be the function taking an orthogonal matrix to its column space, the uniform distribution $P_{\mathcal{G}(k,p)}$ on $\mathcal{G}(k,p)$ is the pushforward measure of $P_{\mathcal{V}(k,p)}$ under $l.$ In other words, the measure of $A \subseteq \mathcal{G}(k,p)$ is $P_{\mathcal{G}(k,p)}[A] = P_{\mathcal{V}(k,p)}[l^{-1}(A)].$ The uniform distributions on these manifolds have a long history in probability, as we discuss in Section \ref{normal_apprx_section}, and in statistics, where they appear in foundational work on multivariate theory \citep{James1954}.

Non-uniform distributions on the Stiefel and Grassmann manifolds play an important role in modern statistical applications. They are used to model directions, axes, planes, rotations, and other data lying on compact Riemannian manifolds in the field of directional statistics \citep{Mardia2009}. Also, statistical models having the Stiefel or Grassmann manifold as their parameter space are increasingly common \citep{Hoff2007, Hoff2009, Cook2010}. In particular, Bayesian analyses of multivariate data often involve prior and posterior distributions on $\mathcal{V}(k,p)$ or $\mathcal{G}(k,p).$ Bayesian inference typically requires simulating from these posterior distributions, motivating the development of new Markov chain Monte Carlo (MCMC) methodology \citep{Hoff2009, Byrne2013, Rao2016}.

The challenge of performing calculations with random orthogonal matrices has motivated researchers to parametrize sets of square orthogonal matrices in terms of Euclidean parameters. We provide a few examples. In what \citet{Diaconis2017} identify as the earliest substantial mathematical contribution to modern random matrix theory, \citet{Hurwitz1897} parametrizes the special orthogonal and unitary groups using Euler angles and computes the volumes of their invariant measures. An implication of these computations is that the Euler angles of a uniformly distributed matrix follow independent beta distributions. \citet{Anderson1987} discuss the potential of various parametrizations in the simulation of uniformly distributed square orthogonal matrices. Other authors have made use of parametrizations of square orthogonal or rotation matrices in statistical applications \citep{Leon2006, Cron2016}.

In contrast, the topic of parametrizing random orthogonal matrices having fewer columns than rows has received little attention. The recent work of \citet{Shepard2015} extends four existing approaches to parametrizing square orthogonal matrices to the case when $k < p$ and to the scenario in which only the column space of the orthogonal matrix is of interest. Naturally, this latter scenario is closely related to the Grassmann manifold. The tools needed to use these parametrizations in a probabilistic setting are still largely missing.

In this article, we lay foundations for application of the Cayley parametrization of the Stiefel and Grassmann manifolds in a probabilistic setting. There are three main contributions. First, we elaborate and expand upon the work of \citet{Shepard2015} to show that the Grassmann manifold $\mathcal{G}(k,p)$ can be represented by the subset $\mathcal{V}^{+}(k,p)$ of orthogonal matrices having an SPD top block and that this subset can be parametrized in terms of Euclidean elements using the Cayley transform. Next, we derive the necessary Jacobian terms for change of variables formulas. Given a density defined on $\mathcal{V}(k,p)$ or $\mathcal{V}^{+}(k,p)$, these allow us to specify the corresponding density for the Euclidean parameters, and vice versa. As an application, we describe and illustrate through examples an approach to MCMC simulation from distributions on these sets. Finally, we establish an asymptotic independent normal approximation for the distribution of the Euclidean parameters which corresponds to the uniform distribution on the Stiefel manifold. This result contributes to the growing literature on normal approximations to the entries of random orthogonal matrices or transformations thereof.

Code to replicate the figures in this article and to simulate from distributions on $\mathcal{V}(k,p)$ and $\mathcal{V}^{+}(k,p)$ is available at \url{https://github.com/michaeljauch/cayley}.

\section{Probability distributions on submanifolds of $\mathbbm{R}^n$} \label{prob_manifolds}

In this section, we introduce tools for defining and manipulating probability distributions on a $m$-dimensional submanifold $\mathcal{M}$ of $\mathbbm{R}^n$. In particular, we discuss the reference measure with respect to which we define densities on $\mathcal{M},$ we make precise what it means for us to parametrize $\mathcal{M}$ in terms of Euclidean parameters, and we state a change of variables formula applicable in this setting. Our formulation of these ideas follows that of \citet{Diaconis2013} somewhat closely. This general discussion will form the basis for our handling of the specific cases in which $\mathcal{M}$ is $\mathcal{V}(k,p)$ or $\mathcal{V}^{+}(k,p).$

In order to specify probability distributions on $\mathcal{M}$ in terms of density functions, we need a reference measure on that space, analogous to Lebesgue measure $L^{m}$ on $\mathbbm{R}^m.$ As in \citet{Diaconis2013} and \citet{Byrne2013}, we take the Hausdorff measure as our reference measure. Heuristically, the $m$-dimensional Hausdorff measure of $A \subset \mathbbm{R}^n$ is the $m$-dimensional area of $A.$ More formally, the $m$-dimensional Hausdorff measure $H^m(A)$ of $A$ is defined 
\begin{align*}
    H^m(A) &= \lim_{\delta \to 0} 
    \inf_{\substack{
    A \subset \cup_i S_i \\  
    \text{diam}(S_i) < \delta
    }} \sum_{i} \alpha_m \left(\frac{\text{diam}(S_i)}{2}\right)^m
\end{align*} where the infimum is taken over countable coverings $\{S_i\}_{i\in \mathbbm{N}}$ of $A$ with $$\text{diam}(S_i) =\sup\left\{|x-y|: x,y \in S_i\right\}$$ and $\alpha_m = \Gamma(\frac12)^m/\Gamma(\frac{m}{2}+1),$ the volume of the unit ball in $\mathbbm{R}^n.$ The $d_\mathcal{V}$-dimensional Hausdorff measure on $\mathcal{V}(k,p)$ coincides with $P_{\mathcal{V}(k,p)}$
up to a multiplicative constant.

Let $g: \mathcal{M} \to \mathbbm{R}$ be proportional to a density with respect to the $m$-dimensional Hausdorff measure on $\mathcal{M}.$ Furthermore, suppose $\mathcal{M}$ can be parametrized by a function $f$ from an open domain $\mathcal{D} \in \mathbbm{R}^m$ to $\mathcal{I}=f(\mathcal{D}) \subseteq \mathcal{M}$ satisfying the  following conditions:
\begin{enumerate}
    \item Almost all of $\mathcal{M}$ is contained in the image $\mathcal{I}$ of $\mathcal{D}$ under $f$ so that  $H^m(\mathcal{M}\setminus\mathcal{I}) = 0;$
    \item The function $f$ is injective on $\mathcal{D};$
    \item The function $f$ is continuously differentiable on $\mathcal{D}$ with the derivative matrix $Df(\bm{\phi})$ at $\bm{\phi} \in \mathcal{D}.$
\end{enumerate} In this setting, we obtain a simple change of variables formula. Define the $m$-dimensional Jacobian of $f$ at $\bm{\phi}$ as $J_mf(\bm{\phi})=\left|Df(\bm{\phi})^T Df(\bm{\phi})\right|^{1/2}.$ Like the familiar Jacobian determinant, this term acts as a scaling factor in a change of variables formula. However, it is defined even when the derivative matrix is not square. For more details, see the discussion in \citet{Diaconis2013}. The change of variables formula is given in the following theorem, which is essentially a restatement of the main result of \citet{Traynor1993}: 
\begin{theorem} \label{area_formula}
 For all Borel subsets $A \subset \mathcal{D}$ 
\begin{align*}
   \int_{A} J_m f(\bm{\phi})L^m(d\bm{\phi}) = H^m[f(A)]   
\end{align*}
and hence
\begin{align*}
     \int_{A} g[f(\bm{\phi})] J_m f(\bm{\phi})L^m(d\bm{\phi}) = \int_{f(A)} g(\bm{y}) H^m(d\bm{y}).
\end{align*}
\end{theorem} 

Naturally, the change of variables formula has an interpretation in terms of random variables. Let $\bm{y}$ be a random element of $\mathcal{M}$ whose distribution has a density proportional to $g.$ Then $\bm{y} \stackrel{\text{dist.}}{=} f(\bm{\phi})$ when the distribution of $\bm{\phi} \in \mathcal{D}$ has a density proportional to $g[f(\bm{\phi})] J_m f(\bm{\phi}).$ 

\section{The Cayley parametrizations} \label{Cayley}

The Cayley transform, as introduced in \citet{Cayley1846}, is a map from  skew-symmetric matrices to special orthogonal matrices. Given $\bm{X}$ in  $\text{Skew}(p) = \{\bm{X} \in \mathbbm{R}^{p \times p} \vert \bm{X} = -\bm{X}^T\},$ the (original) Cayley transform of $\bm{X}$ is the special orthogonal matrix
$$C_{\text{orig.}}(\bm{X}) = (\bm{I}_p + \bm{X})(\bm{I}_p - \bm{X})^{-1}.$$ We work instead with a modified version of the Cayley transform described in \citet{Shepard2015}. In this version, the Cayley transform of $\bm{X}$ is the $p \times k$ orthogonal matrix $$C(\bm{X}) = (\bm{I}_p + \bm{X})(\bm{I}_p -\bm{X})^{-1}\bm{I}_{p\times k}$$ where $\bm{I}_{p\times k}$ denotes the $p\times k$ matrix having the identity matrix as its top block and the remaining entries zero. The matrix $\bm{I}_p - \bm{X}$ is invertible for any $\bm{X} \in \text{Skew}(p)$, so the Cayley transform is defined everywhere. 

In this section, we parametrize (in the sense of Section \ref{prob_manifolds}) the sets $\mathcal{V}(k,p)$ and $\mathcal{V}^{+}(k,p)$ using the Cayley transform $C.$ We are able to parametrize these distinct sets by restricting the domain of $C$ to distinct subsets of $\text{Skew}(p).$ The third condition of the previous section requires that $C$ be continuously differentiable on its domain. We verify this condition by computing the derivative matrices of $C,$ clearing a path for the statement of change of variables formulas in Section \ref{densities}. We also state and discuss an important proposition which justifies our claim that the Grassmann manifold $\mathcal{G}(k,p)$ can be represented by the set $\mathcal{V}^{+}(k,p).$

\subsection{Cayley parametrization of the Stiefel manifold}   \label{Cayley_Stiefel}
To parametrize the Stiefel manifold $\mathcal{V}(k,p),$ the domain of $C$ is restricted to the subset
\begin{align*}\mathcal{D}_{\mathcal{V}} = 
\left\{
\bm{X} = 
\begin{bmatrix}
\bm{B} & -\bm{A}^T \\ 
\bm{A} & \bm{0}_{p-k}
\end{bmatrix} \, \middle| \, \begin{array}{ll} \bm{B} \in \mathbbm{R}^{k \times k}\in \text{Skew}(k) \\ 
\bm{A} \in \mathbbm{R}^{p-k\times k}  \end{array} \right\} \subset \text{Skew}(p).
\end{align*} Let $\bm{X} \in \mathcal{D}_{\mathcal{V}}$ and set $\bm{Q} = C(\bm{X}).$ Partition $\bm{Q} = [\bm{Q}_1^T \quad \bm{Q}_2^T]^T$ so that $\bm{Q}_1$ is square. We can write the blocks of $\bm{Q}$ in terms of the blocks of $\bm{X}$ as
\begin{align*}
    \bm{Q}_1 &= (\bm{I}_k - \bm{A}^T\bm{A} + \bm{B})(\bm{I}_k + \bm{A}^T\bm{A} - \bm{B})^{-1} \\ 
    \bm{Q}_2 &= 2\bm{A}(\bm{I}_k + \bm{A}^T\bm{A} - \bm{B})^{-1}.
\end{align*} The matrix $\bm{I}_k + \bm{A}^T\bm{A} - \bm{B}$ is the sum of a symmetric positive definite matrix and a skew-symmetric matrix and therefore nonsingular (see the appendix for a proof). This observation guarantees (again) that the Cayley transform is defined for all $\bm{X} \in \mathcal{D}_{\mathcal{V}}.$ We can recover the matrices $\bm{A}$ and $\bm{B}$ from $\bm{Q}:$ 
\begin{align} 
    \bm{F} &= (\bm{I}_k - \bm{Q}_1)(\bm{I}_k + \bm{Q}_1)^{-1}  \label{inverse_eqn_F}\\ 
    \bm{B} &= \frac12 (\bm{F}^T - \bm{F})  \label{inverse_eqn_B}\\ 
    \bm{A} &= \frac12 \bm{Q}_2 (\bm{I}_k + \bm{F}). \label{inverse_eqn_A}
\end{align} 

We are now in a position to verify the first two conditions of Section \ref{prob_manifolds}. 
The image of $\mathcal{D}_{V}$ under $C$ is the set $$\mathcal{I}_{\mathcal{V}} = \{\bm{Q}=[\bm{Q}_1^T \quad \bm{Q}_2^T]^T \in \mathcal{V}(k,p) \mid \bm{I}_k + \bm{Q}_1 \text{ is nonsingular}\}$$ which has measure one with respect to $P_{\mathcal{V}(k,p)}.$ The injectivity of $C$ on $\mathcal{D}_{\mathcal{V}}$ follows from the existence of the inverse mapping described in equations \ref{inverse_eqn_F} - \ref{inverse_eqn_A}. All that remains to be verified is the third condition, that $C$ is continuously differentiable on its domain.

As the first step in computing the derivative matrix of $C,$ we define a $d_{\mathcal{V}}$-dimensional vector $\bm{\varphi}$ containing each of the independent entries of $\bm{X}.$ Let $\bm{b}$ be the $k(k-1)/2$-dimensional vector of independent entries of $\bm{B}$ obtained by eliminating diagonal and supradiagonal elements from the vectorization $\vect{\bm{B}}.$ The vector $\bm{\varphi} = (\bm{b}^T, \vect{\bm{A}}^T)^T$ then contains each of the independent entries of $\bm{X}.$ Let $\bm{X}_{\bm{\varphi}} \in \mathcal{D}_{\mathcal{V}}$ be the matrix having $\bm{\varphi}$ as its corresponding vector of independent entries.

The Cayley transform can now be thought of as a function of $\bm{\varphi}:$ 
\begin{align*}
    C(\bm{\varphi}) &= (\bm{I}_p + \bm{X}_{\bm{\varphi}})(\bm{I}_p - \bm{X}_{\bm{\varphi}})^{-1}\bm{I}_{p\times k}.
\end{align*} As a function of $\bm{\varphi},$ the Cayley transform is a bijection between the set $\mathcal{D}_{\mathcal{V}}^{\bm{\varphi}}=\{\bm{\varphi} \in \mathbbm{R}^{d_{\mathcal{V}}} : \bm{X}_{\bm{\varphi}} \in \mathcal{D}_{\mathcal{V}}\} = \mathbbm{R}^{d_{\mathcal{V}}}$ and $\mathcal{I}_{\mathcal{V}}.$ The inverse Cayley transform is defined in the obvious way as the map $C^{-1}: \bm{Q} \mapsto (\bm{b}^T, \vect{\bm{A}}^T)^T$ where $\bm{B}$ and $\bm{A}$ are computed from $\bm{Q}$ according to equations \ref{inverse_eqn_F}-\ref{inverse_eqn_A} and $\bm{b}$ contains the independent entries of $\bm{B}$ as before.

The next lemma provides an explicit linear map $\bm{\Gamma}_{\mathcal{V}}$ from $\bm{\varphi}$ to $\vect{\bm{X}_{\bm{\varphi}}},$ greatly simplifying our calculation of the derivative matrix $DC(\bm{\varphi}).$ The entries of $\bm{\Gamma}_{\mathcal{V}}$ belong to the set $\{-1, 0, 1\}.$ The construction of $\bm{\Gamma}_{\mathcal{V}}$ involves the commutation matrix $\bm{K}_{p,p}$ and the matrix $\widetilde{\bm{D}}_k$ satisfying ${\widetilde{\bm{D}}_k\bm{b} = \vect{\bm{B}},}$ both of which are discussed in \citet{Magnus1988LS} and defined explicitly in Appendix \ref{special_matrices} . Set $\bm{\Theta}_1 = [\bm{I}_k \quad \bm{0}_{k \times p-k}]$ and $\bm{\Theta}_2 = [\bm{0}_{p-k \times k} \quad \bm{I}_{p-k}].$ 
\begin{lemma} \label{lintran_stiefel} The equation ${\vect{X}_{\bm{\varphi}} = \bm{\Gamma}_{\mathcal{V}}\,\bm{\varphi}}$ is satisfied by the matrix
$${\bm{\Gamma}_{\mathcal{V}} = [(\bm{\Theta}_1^T \otimes \bm{\Theta}_1^T)\widetilde{\bm{D}}_k \quad (\bm{I}_{p^2} - \bm{K}_{p,p})(\bm{\Theta}_1^T \otimes \bm{\Theta}_2^T)]}.$$ 
\end{lemma}

With these pieces in place, we can now identify the derivative matrix $C(\bm{\varphi}).$ 
\begin{proposition} \label{jacobian_stiefel} The Cayley transform is continuously differentiable on $\mathcal{D}_{\mathcal{V}}^{\bm{\varphi}} = \mathbbm{R}^{d_{\mathcal{V}}}$ with derivative matrix 
\begin{align*}
    DC(\bm{\varphi}) = 2\left[\bm{I}_{p\times k}^T(\bm{I}_p - \bm{X}_{\bm{\varphi}})^{-T} \otimes (\bm{I}_p - \bm{X}_{\bm{\varphi}})^{-1}\right] \bm{\Gamma}_{\mathcal{V}}.
\end{align*}
\end{proposition} \noindent The form of the derivative matrix reflects the composite structure of the Cayley transform as a function of $\bm{\varphi}.$ The Kronecker product term arises from differentiating $C$ with respect to $\bm{X}_{\bm{\varphi}},$ while the matrix $\bm{\Gamma}_{\mathcal{V}}$ arises from differentiating $\bm{X}_{\bm{\varphi}}$ with respect to $\bm{\varphi}.$

\subsection{The Cayley parametrization of the Grassmann manifold} \label{Cayley_Grassmann}

While the definition of the Grassmann manifold as a collection  of subspaces is fundamental, we often need a more concrete representation. One simple idea is to represent a subspace $S \in \mathcal{G}(k,p)$ by $\bm{Q} \in \mathcal{V}(k,p)$ having $S$ as its column space. However, the choice of $\bm{Q}$ is far from unique. Alternatively, the subspace $S$ can be represented uniquely by the orthogonal projection matrix onto $S,$ as in \citet{Chikuse2003}. We propose instead to represent $\mathcal{G}(k,p)$ by the subset of orthogonal matrices 
$$\mathcal{V}^{+}(k,p) = \left\{\bm{Q} = [\bm{Q}_1^T \, \bm{Q}_2^T]^T \in \mathcal{V}(k,p)\, \middle| \, \bm{Q}_1 \succ \bm{0}\right\}.$$ As the next proposition makes precise, almost every element of $\mathcal{G}(k,p)$ can be represented uniquely by an element of $\mathcal{V}^{+}(k,p).$ Recall that we defined $l: \mathcal{V}(k,p) \to \mathcal{G}(k,p)$ as the map which sends each element of $\mathcal{V}(k,p)$ to its column space.

\begin{proposition} \label{grassmann_bijection}
The map $l$ is injective on $\mathcal{V}^{+}(k,p)$ and  the image of $\mathcal{V}^{+}(k,p)$ under $l$ has measure one with respect to the uniform probability measure $P_{\mathcal{G}(k,p)}$ on $\mathcal{G}(k,p).$ 
\end{proposition} 

In turn, the set $\mathcal{V}^{+}(k,p)$ is amenable to parametrization by the Cayley transform. In this case, the domain of the Cayley transform $C$ is restricted to the subset
$$\mathcal{D}_{\mathcal{G}} = 
\left\{
\bm{X} = 
\begin{bmatrix}
\bm{0} & -\bm{A}^T \\ 
\bm{A} & \bm{0}_{p-k}
\end{bmatrix}
\middle| \begin{array}{ll} \bm{A} \in \mathbbm{R}^{p-k\times k}, \\ \text{eval}_i(\bm{A}^T\bm{A})  \in [0,1) \text{ for } 1 \leq i \leq k \end{array}
\right\} \subset \text{Skew}(p)$$ where the notation $\text{eval}_i(\bm{A}^T\bm{A})$ indicates the $i$th eigenvalue of the matrix $\bm{A}^T\bm{A}.$ Let $\bm{X} \in \mathcal{D}_{\mathcal{G}}$ and set $\bm{Q} = C(\bm{X}).$ Again, partition $\bm{Q} = [\bm{Q}_1^T \quad \bm{Q}_2^T]^T$ so that $\bm{Q}_1$ is square. We can write the blocks of $\bm{Q}$ in terms of $\bm{A}$ as
\begin{align*}
    \bm{Q}_1 &= (\bm{I}_k - \bm{A}^T\bm{A})(\bm{I}_k + \bm{A}^T\bm{A})^{-1} \\ 
    \bm{Q}_2 &= 2\bm{A}(\bm{I}_k + \bm{A}^T\bm{A})^{-1}.
\end{align*} We can recover the matrix $\bm{A}$ from $\bm{Q}:$
\begin{align}
    \bm{F} &= (\bm{I}_k - \bm{Q}_1)(\bm{I}_k + \bm{Q}_1)^{-1} \label{invgrassF}\\
    \bm{A} &= \frac12 \bm{Q}_2 (\bm{I}_k + \bm{F}).\label{invgrassA}
\end{align}

We turn to verifying the conditions of Section \ref{prob_manifolds}. The first two conditions are satisfied as a consequence of the following proposition:  
\begin{proposition} \label{cayley_1to1}
The Cayley transform $C: \mathcal{D}_{\mathcal{G}} \to \mathcal{G}(k,p)$ is one-to-one. 
\end{proposition} \noindent Only the third condition, that $C$ is continuously differentiable on $\mathcal{D}_{G},$ remains.  

The process of computing the derivative matrix is the same as before. We define a $d_{\mathcal{G}}$-dimensional vector $\bm{\psi}=\vect{\bm{A}}$ containing each of the independent entries of $\bm{X}.$ Let $\bm{X}_{\bm{\psi}} \in \mathcal{D}_{\mathcal{G}}$ be the matrix having $\bm{\psi}$ as its corresponding vector of independent entries. The Cayley transform can now be thought of as a function of $\bm{\psi}:$ 
\begin{align*}
    C(\bm{\psi}) &= (\bm{I}_p + \bm{X}_{\bm{\psi}})(\bm{I}_p - \bm{X}_{\bm{\psi}})^{-1}\bm{I}_{p\times k}. 
\end{align*} As a function of $\bm{\psi},$ the Cayley transform is a bijection between the set $\mathcal{D}_{\mathcal{G}}^{\bm{\psi}}=\{\bm{\psi} \in \mathbbm{R}^{d_{\mathcal{G}}} : \bm{X}_{\bm{\psi}} \in \mathcal{D}_{\mathcal{G}}\}$ and $\mathcal{V}^{+}(k,p).$ The next lemma provides an explicit linear map from $\bm{\psi}$ to $\vect{\bm{X}_{\bm{\psi}}}$ in the form of a $\{-1,0,1\}$ matrix $\bm{\Gamma}_{\mathcal{G}}:$ 
\begin{lemma} \label{lintran_grassmann}
The equation ${\vect{X}_{\bm{\psi}} = \bm{\Gamma}_{\mathcal{G}}\bm{\psi}}$ is satisfied by the matrix $${\bm{\Gamma}_{\mathcal{G}} =  (\bm{I}_{p^2} - \bm{K}_{p,p})(\bm{\Theta}_1^T \otimes \bm{\Theta}_2^T)}.$$ 
\end{lemma} \noindent In the next proposition, we identify the derivative matrix $DC(\phi).$
\begin{proposition}\label{jacobian_grassmann} The Cayley transform is continuously differentiable on $\mathcal{D}_{\mathcal{G}}^{\bm{\psi}}$ with derivative matrix 
\begin{align*}
    DC(\bm{\psi}) = 2\left[\bm{I}_{p\times k}^T(\bm{I}_p - \bm{X}_{\bm{\psi}})^{-T} \otimes (\bm{I}_p - \bm{X}_{\bm{\psi}})^{-1}\right] \bm{\Gamma}_{\mathcal{G}}.
\end{align*}
\end{proposition}

\section{Change of variables formulas} \label{densities}

We now state change of variables formulas for the Cayley parametrizations of $\mathcal{V}(k,p)$ and $\mathcal{V}^{+}(k,p).$ Given the results of the previous section, the formulas follow directly from Theorem \ref{area_formula}. The $d_{\mathcal{V}}$-dimensional Jacobian $J_{d_{\mathcal{V}}}C(\bm{\varphi})$ of the Cayley transform $C$ at $\bm{\varphi}$ is equal to
\begin{align*}
    J_{d_{\mathcal{V}}}C(\bm{\varphi}) &= \left|DC(\bm{\varphi})^T DC(\bm{\varphi})\right|^{1/2} \\ 
    &= \left|2^2\, \bm{\Gamma}_{\mathcal{V}}^T \,(\bm{G}_{\mathcal{V}} \otimes \bm{H}_{\mathcal{V}})\, \bm{\Gamma}_{\mathcal{V}}\right|^{1/2}
\end{align*} where 
\begin{align*}
    \bm{G}_{\mathcal{V}} &= (\bm{I}_p - \bm{X}_{\bm{\varphi}})^{-1}\bm{I}_{p\times k}\bm{I}_{p\times k}^T (\bm{I}_p - \bm{X}_{\bm{\varphi}})^{-T}\\ 
    \bm{H}_{\mathcal{V}} &= (\bm{I}_p - \bm{X}_{\bm{\varphi}})^{-T}(\bm{I}_p - \bm{X}_{\bm{\varphi}})^{-1}
\end{align*} and $\bm{\Gamma}_{\mathcal{V}}$ is defined as in Section \ref{Cayley_Stiefel}. The equations above hold for the $d_{\mathcal{G}}$-dimensional Jacobian $J_{d_{\mathcal{G}}}C(\bm{\phi})$ if we replace the symbols $\mathcal{V}$ and $\bm{\varphi}$ with the symbols $\mathcal{G}$ and $\bm{\psi},$ respectively.  In the supplementary material, we describe how to compute the Jacobian terms, taking advantage of their block structure. Let $g: \mathcal{V}(k,p) \to \mathbbm{R}$ be proportional to a density with respect to the $d_{\mathcal{V}}$-dimensional Hausdorff measure on $\mathcal{V}(k,p).$ 
\begin{theorem} (Change of Variables Formulas) \label{Euclidean_densities} For all Borel sets $A \subset \mathcal{D}_{\mathcal{V}}^{\bm{\varphi}}$  
\begin{align}
      \int_{A} J_{d_{\mathcal{V}}} C(\bm{\varphi})L^{d_{\mathcal{V}}}(d\bm{\varphi}) = H^{d_{\mathcal{V}}}[C(A)]
\end{align} and hence
\begin{align}
     \int_{A} g\left[C(\bm{\varphi})\right] J_{d_{\mathcal{V}}} C(\bm{\varphi})L^{d_{\mathcal{V}}}(d\bm{\varphi}) = \int_{C(A)} g(\bm{Q}) H^{d_{\mathcal{V}}}(d\bm{Q}). 
\end{align} If instead we have $g: \mathcal{V}^{+}(k,p) \to \mathbbm{R}$ proportional to a density with respect to the $d_{\mathcal{G}}$-dimensional Hausdorff measure on $\mathcal{V}^{+}(k,p),$ the statement is true when we  replace the symbols $\mathcal{V}$ and $\bm{\varphi}$ with the symbols $\mathcal{G}$ and $\bm{\psi},$ respectively.
\end{theorem}

Similarly to Theorem \ref{area_formula}, Theorem \ref{Euclidean_densities} has an interpretation in terms of random variables. Let the distribution of $\bm{Q} \in \mathcal{V}(k,p)$ have a density proportional to $g.$ Then $\bm{Q} \stackrel{\text{dist.}}{=} C(\bm{\varphi})$ when the distribution of $\bm{\varphi} \in \mathcal{D}_{\mathcal{V}}^{\bm{\varphi}} = \mathbbm{R}^{d_{\mathcal{V}}}$ has a density proportional to $g\left[C(\bm{\varphi})\right] J_{d_{\mathcal{V}}} C(\bm{\varphi}).$ In particular, let $g \propto 1$ so that $\bm{Q} \sim P_{\mathcal{V}(k,p)}.$ Then $\bm{Q} \stackrel{\text{dist.}}{=} C(\bm{\varphi})$ when the distribution of $\bm{\varphi}$ has a density proportional to $J_{d_{\mathcal{V}}} C(\bm{\varphi}).$ Analogous statements hold when $\bm{Q}$ is a random element of $\mathcal{V}^{+}(k,p.)$
 
\section{Simulating from $\mathcal{V}(k,p)$ and $\mathcal{V}^{+}(k,p)$} \label{simulation_section}

Practical applications often require simulating a random orthogonal matrix $\bm{Q}$ whose distribution has a prescribed density $g.$ For instance, Bayesian analyses of statistical models with an orthogonal matrix parameter yield posterior densities on the Stiefel manifold, and inference typically requires simulating from these densities. In many cases, generating independent samples is too challenging and MCMC methods are the most sensible option. 

In this section, we present an MCMC approach to simulating from a distribution having a density on the set $\mathcal{V}(k,p)$ or $\mathcal{V}^{+}(k,p)$ which takes advantage of the Cayley parametrizations described in Section \ref{Cayley}. The recent work of \citet{Pourzanjani2017} explores a similar idea based on a Givens rotation parametrization of the Stiefel manifold. When it is not too computationally expensive, our approach may have certain advantages over existing methods. Unlike \citet{Hoff2009}, it can be applied regardless of whether conditional distributions belong to a particular parametric family, and it is arguably simpler to implement than the approach of \citet{Byrne2013}. In statistical applications where interest lies in a subspace rather than a particular orthogonal basis, our representation of the Grassmann manifold $\mathcal{G}(k,p)$ by the set $\mathcal{V}^{+}(k,p)$ may suggest an appealing parametrization, with the MCMC approach of this section offering the machinery for Bayesian inference. 

We illustrate the basic idea with the Stiefel manifold. (Simulating from $\mathcal{V}^{+}(k,p)$ involves analogous steps.) In order to simulate $\bm{Q}$ whose distribution has density $g$ on the Stiefel manifold $\mathcal{V}(k,p),$ we construct a Markov chain whose stationary distribution has density $g\left[C(\bm{\varphi})\right] J_{d_{\mathcal{V}}} C(\bm{\varphi})$ on the set $\mathcal{D}_{\mathcal{V}}^{\bm{\varphi}} = \mathbbm{R}^{d_{\mathcal{V}}}.$ Then we simply transform the realized Markov chain back to $\mathcal{V}(k,p)$ using the Cayley transform. By doing so, we avoid the difficulty of choosing and simulating from an efficient proposal distribution defined on the Stiefel manifold.

To make things more concrete, we describe the approach with the Metropolis-Hastings algorithm as our MCMC method. We start with an initial value $\bm{\varphi}_0 \in \mathcal{D}_{\mathcal{V}}^{\bm{\varphi}}$ for our chain and a density $q(\bm{\varphi}' \vert \bm{\varphi})$ for the proposal $\bm{\varphi}'$ given the previous value $\bm{\varphi}.$ The Metropolis-Hastings algorithm for simulating from the distribution having density $g\left[C(\bm{\varphi})\right] J_{d_{\mathcal{V}}} C(\bm{\varphi})$ on $\mathcal{D}_{\mathcal{V}}^{\bm{\varphi}}$ proceeds as follows. For $t=0, ..., T:$
\begin{enumerate}
    \item Generate $\bm{\varphi}'$ from $q(\bm{\varphi}' \vert \bm{\varphi}_t).$
    \item Compute the acceptance ratio 
    \begin{align*}
        r &= \frac{g\left[C(\bm{\varphi}')\right]J_{d_{\mathcal{V}}}C(\bm{\varphi}')}{g\left[C(\bm{\varphi}_t)\right]J_{d_{\mathcal{V}}}C(\bm{\varphi}_t)}\frac{q(\bm{\varphi}_t \vert \bm{\varphi}')}{q(\bm{\varphi}' \vert \bm{\varphi}_t)}.
    \end{align*}
    \item Sample $u \sim \text{Unif}(0,1).$ If $u \leq r,$ set $\bm{\varphi}_{t+1} = \bm{\varphi}'.$ Otherwise, set $\bm{\varphi}_{t+1} = \bm{\varphi}_{t}.$
\end{enumerate} For a broad class of $g$ and $q,$ the orthogonal matrices $\{C(\bm{\varphi}_{t})\}_{t=0}^{T}$ approximate the distribution having density $g$ when $T$ is large enough.

In place of this simple Metropolis-Hastings algorithm, we can substitute other MCMC methods. Hamiltonian Monte Carlo (HMC) \citep{Neal2010}, with implementations in software such as \citet{Carpenter2017} and \citet{Salvatier2016}, is a natural choice. For settings in which evaluation and automatic differentiation of the Jacobian term are not prohibitively expensive, the Cayley transform approach offers a relatively straightforward path to MCMC simulation on the sets $\mathcal{V}(k,p)$ and $\mathcal{V}^{+}(k,p).$

\subsection{Example: The uniform distribution on $\mathcal{V}(k,p)$}

\begin{figure}
\centerline{\includegraphics[width=5.5in]{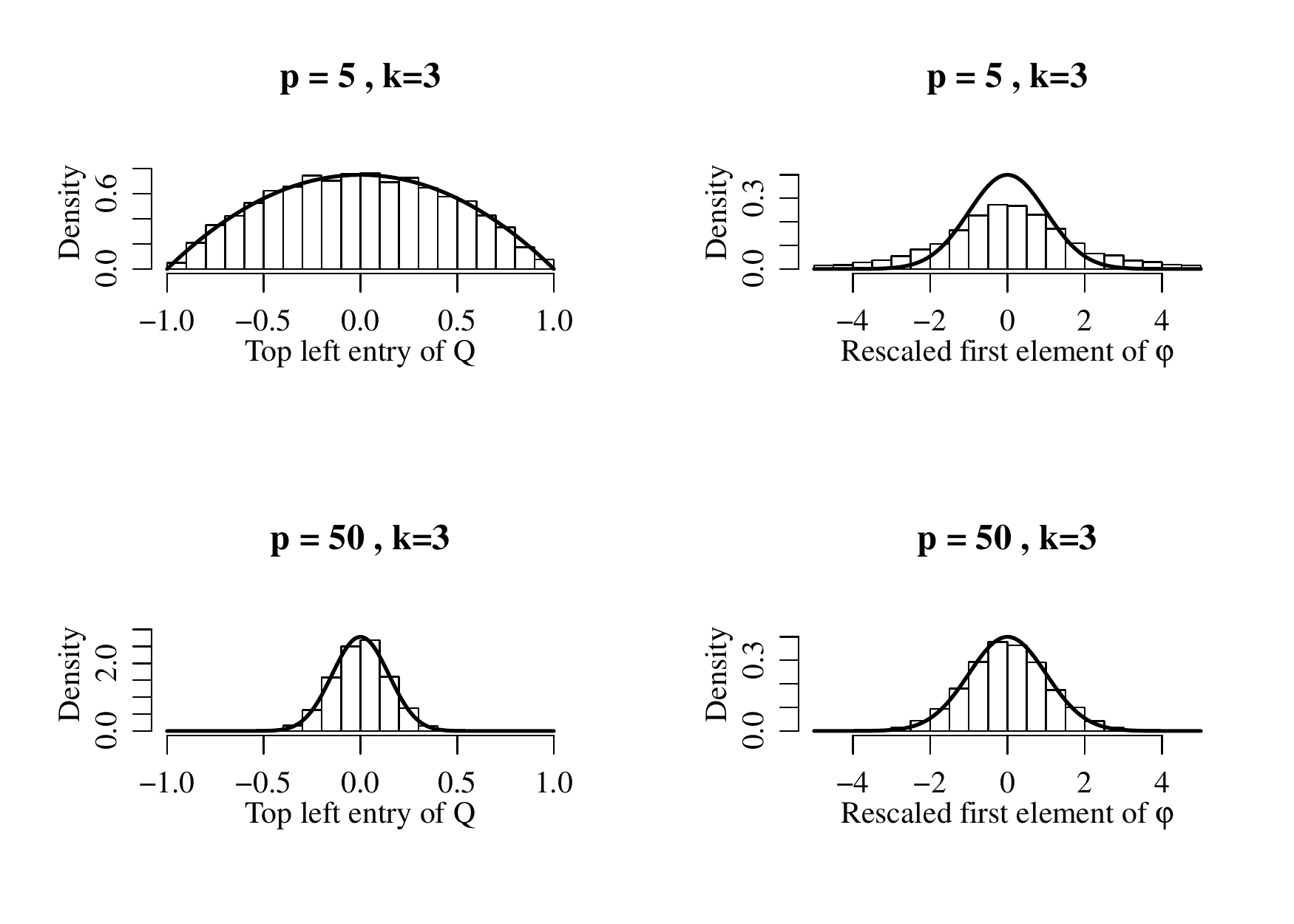}}
\caption{The top left panel is a histogram of the top left entry of simulated values of $\bm{Q} \sim P_{\mathcal{V}(3,5)}$ plotted against the exact density. The bottom left panel is the analogous plot for $\bm{Q} \sim P_{\mathcal{V}(3,50)}.$ The top right panel is a histogram of the first entry, rescaled by $\sqrt{p/2},$ of $\bm{\varphi} = C^{-1}(\bm{Q})$ when $p=5$ and $k=3.$ The histogram is plotted against a standard normal density. The bottom right panel is the analogous plot for the case when $p=50$ and $k=3.$}
\label{fig:uniform_sim}
\end{figure}

Using the MCMC approach described above, we simulate from the uniform distribution on $\mathcal{V}(k,p).$ Specifically, we use HMC as implemented in Stan \citep{Carpenter2017} to simulate $\bm{Q} \sim P_{\mathcal{V}(k,p)}.$ Of course, there exist many algorithms for the simulation of independent, uniformly distributed orthogonal matrices. The uniform distribution only serves as a starting point for illustrating the proposed approach.  

Figure \ref{fig:uniform_sim} provides plots based on $10,000$ simulated values of $\bm{Q} \sim P_{\mathcal{V}(k,p)}.$ The top row of the figure deals with the case $p=5$ and $k=3,$ while the bottom row deals with the case $p=50$ and $k=3.$ The histograms on the left show the top left entry of the simulated values of $\bm{Q}$ plotted against the exact density, given in Proposition 7.3 of \citet{Eaton1989}. As we expect, there is close agreement between the histogram density estimate and the exact density. The histograms on the right show the first entry, rescaled by $\sqrt{p/2},$ of the simulated values of the vector $\bm{\varphi} = C^{-1}(Q).$ These are plotted against a standard normal density. Theorem \ref{normal_apprx_theorem} tells us that the histogram density estimate and the standard normal density should agree when $p$ is large (both in an absolute sense and relative to k), which is what we observe in the plot on the bottom right. When $k$ is similar in magnitude to $p,$ the standard normal density is a poor approximation, as we see in the plot on the top right. 

\subsection{Example: Bayesian inference for the spiked covariance model}

\begin{figure}
\centerline{\includegraphics[width=5.5in]{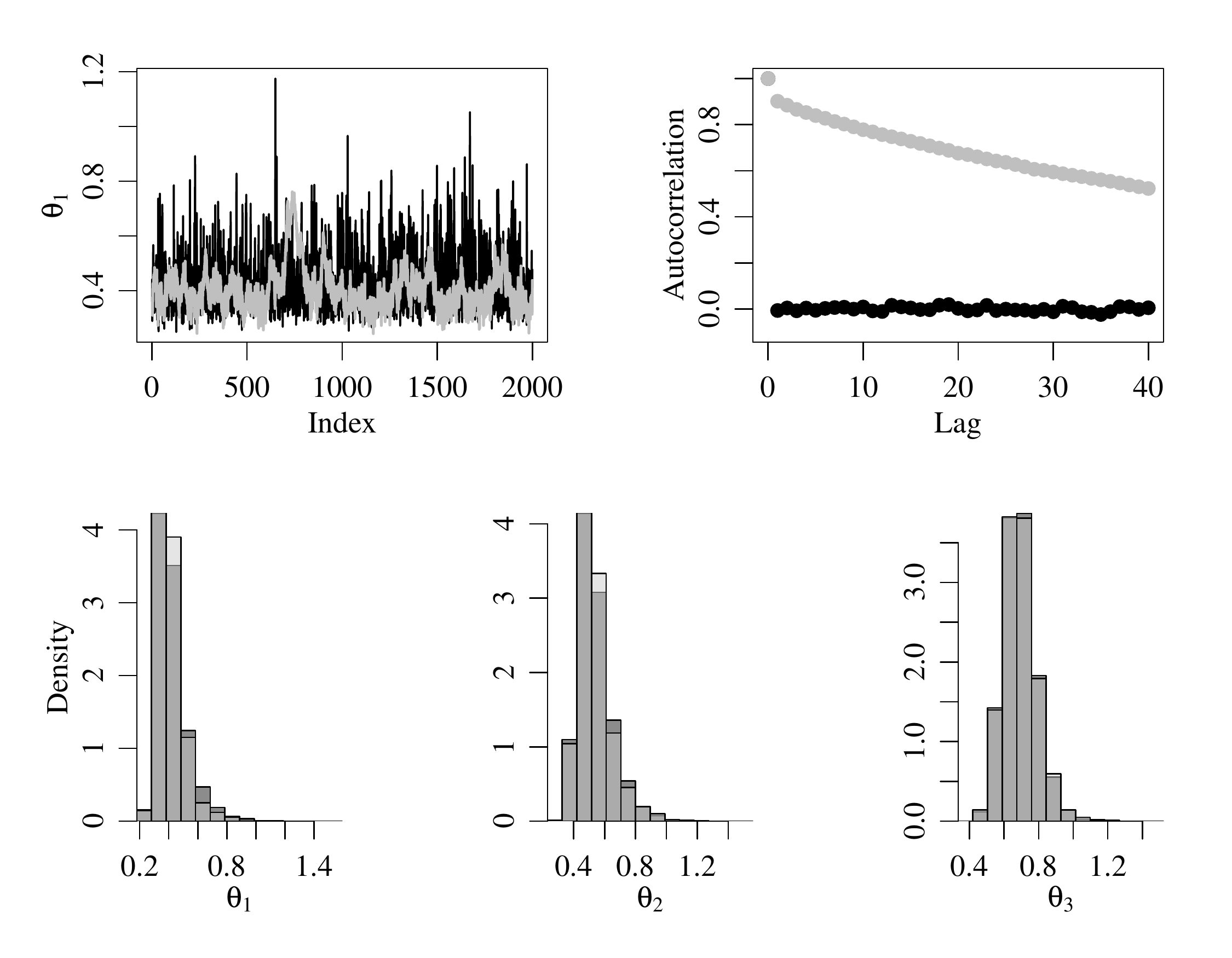}}
\caption{The plot in the top left overlays trace plots of the first principal angle calculated from a portion of each of the chains, while the plot in the top right shows the correlation between lagged values of the first principal angle. The black lines and dots correspond to our MCMC approach, while the gray lines and dots correspond to that of \citet{Hoff2009}. The plots in the bottom half compare histogram approximations of the marginal posterior distributions of the principal angles. }
\label{fig:Binghamsim}
\end{figure}
Suppose the rows of an $n \times p$ data matrix $\bm{Y}$ are independent samples from a mean zero multivariate normal population with covariance matrix $\bm{\Sigma}.$ The spiked covariance model, considered by \citet{Johnstone2001} and others, assumes the covariance matrix $\bm{\Sigma}$ can be decomposed as $\bm{\Sigma} = \sigma^2\left(\bm{Q}\bm{\Lambda}\bm{Q}^T + \bm{I}_p\right)$ with $\bm{Q}\in \mathcal{V}(k,p)$ and $\bm{\Lambda} = \text{diag}(\lambda_1, ..., \lambda_k)$ where $\lambda_1 > ... > \lambda_k > 0.$ Under this model, the covariance matrix is partitioned into the low rank ``signal" component $\sigma^2\bm{Q}\bm{\Lambda}\bm{Q}^T$ and the isotropic ``noise" component $\sigma^2\bm{I}_p.$

Given priors for the unknown parameters, a conventional Bayesian analysis will approximate the posterior distribution by a Markov chain having the posterior as its stationary distribution. Inference for the trivially constrained parameters $\sigma^2$ and $\bm{\Lambda}$ is easily handled by standard MCMC approaches, so we treat these parameters as fixed and focus on inference for the orthogonal matrix parameter $\bm{Q}.$ With a uniform prior for $\bm{Q},$ the posterior distribution is matrix Bingham \citep{Hoff2009}, having density 
\begin{align*}
p(\bm{Q} \mid \bm{Y}, \sigma^2, \bm{\Lambda}) &\propto \etr\left\{\left[\left(\bm{\Lambda}^{-1} + \bm{I}_k\right)^{-1}/\left(2\sigma^2\right)\right]\bm{Q}^T\left[\bm{Y}^T\bm{Y}\right]\bm{Q}\right\}.
\end{align*} 

We compare two MCMC approaches to simulating from the matrix Bingham distribution: the Gibbs sampling method of \citet{Hoff2009} and the Cayley transform approach (again, with HMC as implemented in Stan). The dimensions are chosen as $n=100,\, p=50,$ and $k=3.$ We set $\sigma^2 = 1,\,\bm{\Lambda} = \text{diag}(5, 3, 1.5),$ and choose a true value of $\bm{Q}$ uniformly from $\mathcal{V}(3,50).$ We then generate a data matrix according to the model $\vect{\bm{Y}} \sim \mathcal{N}\left[\bm{0},\sigma^2\left(\bm{Q}\bm{\Lambda}\bm{Q}^T + \bm{I}_p\right) \otimes \bm{I}_{100}\right].$ We run each Markov chain for 12,000 steps and discard the first 2000 steps as burn-in. In order to summarize the high-dimensional posterior simulations in terms of lower dimensional quantities, we compute the principal angles between the columns of the simulated $\bm{Q}$ matrices and the corresponding columns of the posterior mode $\bm{V},$ computed from the eigendecomposition $\bm{A} = \bm{V}\bm{D}\bm{V}^T.$  For $j=1 ,...,3,$ the principal angles are
\begin{align*}
    \theta_{j} = \cos^{-1}\left(\frac{|\bm{q}_{j}^T\bm{v}_j|}{\|\bm{q}_{j}\| \|\bm{v}_j\|}\right)
\end{align*} where $\bm{q}_j$ and $\bm{v}_j$ are the $j$th columns of $\bm{Q}$ and $\bm{V},$ respectively.

Figure \ref{fig:Binghamsim} displays the principal angles calculated from the two Markov chains. The plots in the bottom half of the figure compare histogram approximations of the marginal posterior distributions of the principal angles. There is considerable overlap,  suggesting that the two chains have found their way to equivalent modes of their matrix Bingham stationary distribution. The plot in the top left of the figure overlays trace plots of the first principal angle calculated from a portion of each of the chains. The black line corresponds to our MCMC approach, while the gray line corresponds to that of \citet{Hoff2009}. The plot in the top right shows the correlation between lagged values of the first principal angle. Again, the black dots correspond to our MCMC approach, while the gray dots correspond to that of \citet{Hoff2009}. Together, the plots in the top half of the figure indicate that, compared to the approach of \citet{Hoff2009}, the Cayley transform approach produces a Markov chain with less autocorrelation, reducing Monte Carlo error in the resulting posterior inferences. 

\section{An asymptotic independent normal approximation} \label{normal_apprx_section}

In Section \ref{densities}, we derived the density for the distribution of $\bm{\varphi} = C^{-1}(\bm{Q})$ when $\bm{Q}$ is distributed uniformly on the Stiefel manifold. However, the expression involves the rather opaque Jacobian term $J_{d_{\mathcal{V}}}C(\bm{\varphi}).$ Instead of analyzing the density function, we can gain insight into the distribution of $\bm{\varphi}$ by other means. The critical observation, evident in simulations, is the following: If $\bm{Q} \sim P_{\mathcal{V}(k,p)}$ with $p$ large and $p \gg k$ then, in some sense, the elements of $\bm{\varphi} = C^{-1}(\bm{Q})$ are approximately independent and normally distributed. Theorem \ref{normal_apprx_theorem} of this section provides a mathematical explanation for this empirical phenomenon. 

In order to understand Theorem \ref{normal_apprx_theorem} and its broader context, it is helpful to review the literature relating to normal approximations to the entries of random orthogonal matrices or transformations thereof. For the sake of consistency and clarity, the notation and formulation of the relevant results have been modified slightly. Let $\{\bm{Q}_p\}$ be a sequence of random orthogonal matrices with each element $\bm{Q}_p$ uniform on $\mathcal{V}(k_p, p).$ The notation $k_p$ indicates that the number of columns may grow with the number of rows. For each $p,$ let $q_p$ be the top left entry of $\bm{Q}_p$ (any other entry would also work). It has long been observed that $q_p$ is approximately normal when $p$ is large. The earliest work in this direction relates to the equivalence of ensembles in statistical mechanics and is due to Mehler \citep{Mehler1866}, Maxwell \citep{Maxwell1875, Maxwell1878}, and Borel \citep{Borel1906}. A theorem of Borel shows that $\text{Pr}(\sqrt{p} q_p \leq x) \to \Phi(x)$ as $p$ grows, where $\Phi$ is the cumulative distribution function of a standard normal random variable. Since then, a large and growing literature on this sort of normal approximation has emerged. A detailed history is given in \citet{Diaconis1987}, while an overview is given in \citet{DAristotile2003} and \citet{Jiang2006}.

Much of this literature is devoted to normal approximations to the joint distribution of entries of random orthogonal matrices. Letting $\bm{q}_p$ be the first column of $\bm{Q}_p$ for each $p$ (again, any other column would also work), \citet{Stam1982} proved that the total variation distance between the distribution of the first $m_p$ coordinates of $\bm{q}_p$ and the distribution of $m_p$ standard normal random variables converges to zero as $p$ gets large so long as $m_p=o(\sqrt{p}).$ \citet{Diaconis1987} strengthened this result, showing that it holds for $m_p = o(p).$ \citet{Diaconis1992} prove that the total variation distance between the distribution of the top left $m_p \times n_p$ block of $\bm{Q}_p$ and the distribution of $m_p n_p$ independent standard normals goes to zero as $p \to \infty$ if $m_p = o(p^{\gamma})$ and $n_p = o(p^{\gamma})$ for $\gamma = 1/3.$ (Clearly, we must have $n_p \leq k_p$ for this result to make sense.) Their work drew attention to the problem of determining the largest orders of $m_p$ and $n_p$ such that the total variation distance goes to zero. \citet{Jiang2006} solved this problem, finding the largest orders to be $o(p^{1/2}).$ Recent work has further explored this topic \citep{Stewart2018, Jiang2017}.

Many authors have also considered transformations of random orthogonal matrices or notions of approximation not based on total variation distance. In the former category, \citet{DAristotile2003} and \citet{Meckes2008} study the convergence of linear combinations of the entries of the matrices in the sequence $\{\bm{Q}_p\}$ to normality as $p \to \infty.$ \citet{Diaconis1994, Stein1995, Johansson1997}, and \citet{Rains1997} addresses the normality of traces of powers of random orthogonal and unitary matrices. In the latter category, \citet{Chatterjee2008} and \citet{Jiang2017} consider probability metrics other than total variation distance. \citet{Jiang2006} also considers a notion of approximation other than total variation distance, and Theorem 3 in that work is particularly important in understanding our Theorem \ref{normal_apprx_theorem}.

Theorem 3 of \citet{Jiang2006} tells us that the distribution $P_{\mathcal{V}(k_p,p)}$ can be approximated by the distribution of $pk_p$ independent normals provided that $p$ is sufficiently large (both in an absolute sense and relative to $k_p$). The form of approximation in the theorem is weaker and likely less familiar than one based on total variation distance. Define the max norm $\|\cdot\|_{\text{max}}$ of a matrix as the maximum of the absolute values of its entries. \citet{Jiang2006} shows that one can construct a sequence of pairs of random matrices $\{\bm{Z}_p, \bm{Q}_p\}$ with each pair defined on the same probability space such that 
\begin{enumerate}[(i)]
    \item The entries of the $p \times k_p$ matrix $\bm{Z}_p$ are independent standard normals;
    \item The matrix $\bm{Q}_p$ is uniform on $\mathcal{V}(k_p, p);$
    \item The quantity $\epsilon_p = \|\sqrt{p}\bm{Q}_p - \bm{Z}_p\|_{\text{max}} \to 0$ in probability as $p$ grows provided that $k_p = o(p /\log p),$ and this is the largest order of $k_p$ such that the result holds.
\end{enumerate} The coupling is constructed by letting $\bm{Q}_p$ be the result of the Gram-Schmidt orthogonalization procedure applied to $\bm{Z}_p.$ 

To better understand this type of approximation, which we refer to as `approximation in probability,' consider the problem of simulating from the distribution of the random matrix $\sqrt{p}\bm{Q}_p$ on a computer with finite precision. One could simulate a matrix $\bm{Z}_p$ of independent standard normals, obtain $\bm{Q}_p$ using Gram-Schmidt, then multiply by $\sqrt{p}$ to arrive at $\sqrt{p}\bm{Q}_p.$ However, for a fixed machine precision and $p$ sufficiently large (again, both in an absolute sense and relative to $k_p$), the matrix $\sqrt{p}\bm{Q}_p$ would be indistinguishable from $\bm{Z}_p$ with high probability.

Our Theorem \ref{normal_apprx_theorem} establishes that the distribution of $\bm{\varphi}_p = C_p^{-1}(\bm{Q}_p),$ which we know to have a density proportional to $J_{d_{\mathcal{V}}}C_p(\bm{\varphi}),$ can be approximated in probability by independent normals. (Since we now have a sequence of matrices of different dimensions, we denote the Cayley transform parametrizing $\mathcal{V}(k_p, p)$ by $C_p.$) For each $p,$ define the diagonal scale matrix 
\begin{align*}\bm{\Pi}_p &= 
\begin{bmatrix} \sqrt{p/2}\bm{I}_{k_p(k_p-1)/2} & \bm{0} \\ 
                \bm{0} & \sqrt{p}\bm{I}_{(p-k_p)k_p}
\end{bmatrix},
\end{align*} and recall that the infinity norm $\|\cdot\|_{\infty}$ of a vector is equal to the maximum of the absolute values of its entries.
\begin{theorem}\label{normal_apprx_theorem}
One can construct a sequence of pairs of random vectors $\{\bm{z}_p, \bm{\varphi}_p\}$ such that 
\begin{enumerate}[(i)]
\item The entries of the vector $\bm{z}_p$ are independent standard normals; 
\item The vector $\bm{\varphi}_p \stackrel{\text{dist.}}{=} C_p^{-1}(\bm{Q}_p)$ where $\bm{Q}_p \sim P_{\mathcal{V}(k_p,p)};$
\item The quantity $\epsilon_p := \|\bm{\Pi}_p\bm{\varphi}_p - \bm{z}_p\|_{\infty} \to 0$ in probability as $p \to \infty$ provided $k_p  = o\left(\frac{p^{1/4}}{\sqrt{\log{p}}}\right).$
\end{enumerate}
\end{theorem} 

The construction of the coupling is more elaborate than in Theorem 3 of \citet{Jiang2006}. We first introduce a function $\widetilde{C}_p^{-1}$ which approximates the inverse Cayley transform. Given a matrix $\bm{M} \in \mathbbm{R}^{p \times k_p}$ having a square top block $\bm{M}_{1}$ and a bottom block $\bm{M}_{2}$, the vector $\widetilde{b}_p(\bm{M})$ contains the independent entries of ${\widetilde{B}_p(\bm{M}) = \bm{M}_{1} - \bm{M}_{1}^T}$ obtained by eliminating diagonal and supradiagonal entries from $\vect{\widetilde{B}_p(\bm{M})}$ while $\widetilde{A}_p(\bm{M}) = \bm{M}_{2}.$ The approximate inverse Cayley transform is then 
\begin{align*}
    \widetilde{C}_p^{-1}(\bm{M}) &=
    \begin{bmatrix}
        \widetilde{b}_p(\bm{M}) \\ 
         \vect{\widetilde{A}_p(\bm{M})}
    \end{bmatrix}.
\end{align*} Now let $\bm{Z}_{p}$ be a $p \times k_p$ matrix of independent standard normals and let $\bm{Q}_{p}$ be the result of applying the Gram-Schmidt orthogonalization procedure to $\bm{Z}_{p}.$ It follows that $\bm{Q}_{p} \sim P_{\mathcal{V}(k_p, p)}.$ Finally, set   
\begin{align*}
\bm{z}_p  &= \bm{\Pi}_p \,  \widetilde{C}_p^{-1}(p^{-1/2}\bm{Z}_{p})\\ 
\bm{\varphi}_p &= C_p^{-1}(\bm{Q}_{p}). 
\end{align*} 

The details of the proof of Theorem \ref{normal_apprx_theorem} appear in the appendix, but we provide a sketch here. Part (i) is straightforward to verify and (ii) is immediate. Part (iii) requires more work. The proof of (iii) involves verifying the following proposition, which involves a third random vector $\widetilde{\bm{\varphi}}_p = \widetilde{C}_p^{-1}(\bm{Q}_p):$

\begin{proposition} \label{normal_apprx_prop} (i) The quantity $\|\bm{\Pi}_p\widetilde{\bm{\varphi}}_p - \bm{\Pi}_p\bm{\varphi}_p\|_{\infty} \to 0$ in probability as $p \to \infty$ provided $k_p  = o\left(\frac{p^{1/4}}{\sqrt{\log{p}}}\right),$ and (ii) the quantity $\|\bm{\Pi}_p\widetilde{\bm{\varphi}}_p - \bm{z}_p\|_{\infty} \to 0$ in probability as $p \to \infty$ provided $k_p  = o\left(\frac{p}{\log{p}}\right).$ 
\end{proposition} \noindent Part (iii) then follows from the proposition by the triangle inequality.

\section*{Acknowledgements}
This work was partially supported by grants from the National Science Foundation (DMS-1505136, IIS-1546130) and the United States Office of Naval Research (N00014-14-1-0245/N00014-16-1-2147).
\bibliographystyle{plainnat}
\bibliography{references}

\appendix

\section{Proofs}

\subsection{The sum of a symmetric positive definite matrix and skew-symmetric matrix is nonsingular} \label{sum_pd_skew} 

Let $\bm{\Sigma}$ and $\bm{S}$ be symmetric positive definite and skew-symmetric, respectively, of equal dimension.  Their sum can be written 
\begin{align*}
    \bm{\Sigma} + \bm{S} &= \bm{\Sigma}^{1/2} \left( I + \bm{\Sigma}^{-1/2} \bm{S} \bm{\Sigma}^{-1/2}\right) \bm{\Sigma}^{1/2}.
\end{align*} The matrix $\bm{\Sigma}^{-1/2} \bm{S} \bm{\Sigma}^{-1/2}$ is skew-symmetric, which implies that $\bm{I} + \bm{\Sigma}^{-1/2} \bm{S} \bm{\Sigma}^{-1/2}$ is nonsingular. Because it can be written as the product of nonsingular matrices, the sum $\bm{\Sigma} + \bm{S}$ is nonsingular. 

\subsection{Proof of Proposition \ref{jacobian_stiefel}}

One can compute, following \citet{Magnus1988}, that 
\begin{align*}
    d\bm{Q} = 2(\bm{I}_p - \bm{X}_{\bm{\varphi}})^{-1} d\bm{X}_{\bm{\varphi}} (\bm{I}_p - \bm{X}_{\bm{\varphi}})^{-1} \bm{I}_{p\times k},
\end{align*}
so that 
\begin{align*}
    d\vect{\bm{Q}} = 2\left[\bm{I}_{p\times k}^T(\bm{I}_p - \bm{X}_{\bm{\varphi}})^{-T} \otimes (\bm{I}_p - \bm{X}_{\bm{\varphi}}^{-1})\right] d\vect{\bm{X}_{\bm{\varphi}}}.
\end{align*} By Lemma \ref{lintran_stiefel}, we have that ${\vect{X}_{\bm{\varphi}} = \bm{\Gamma}_{\mathcal{V}}\bm{\varphi}}.$ Thus,
\begin{align*}
    d\vect{\bm{Q}} = 2\left[\bm{I}_{p\times k}^T(\bm{I}_p - \bm{X}_{\bm{\varphi}})^{-T} \otimes (\bm{I}_p - \bm{X}_{\bm{\varphi}})^{-1}\right] \bm{\Gamma}_{\mathcal{V}}\, d\bm{\varphi}.
\end{align*} Using the first identification table of \citet{Magnus1988}, we identify the derivative matrix 
\begin{align*}
    D\text{Cay}_{\mathcal{V}}(\bm{\varphi}) = 2\left[\bm{I}_{p\times k}^T(\bm{I}_p - \bm{X}_{\bm{\varphi}})^{-T} \otimes (\bm{I}_p - \bm{X}_{\bm{\varphi}})^{-1}\right] \bm{\Gamma}_{\mathcal{V}}.
\end{align*} 

\subsection{Proof of Proposition \ref{grassmann_bijection}}
We first show that $l$ is injective on $\mathcal{V}^{+}(k,p).$ Let $\bm{Q} = \left[\bm{Q}_1^T \, \bm{Q}_2^T\right]^T, \, \bm{Q}^{'} = \left[\bm{Q}_1^{'T} \, \bm{Q}_2^{'T}\right]^T \in \mathcal{V}^{+}(k,p)$ and suppose that $l(\bm{Q}) = l\left(\bm{Q}^{'}\right),$ i.e. the columns of $\bm{Q}$ and $\bm{Q}^{'}$ span the same subspace. There must exist $\bm{R}\in \mathcal{O}(k)$ such that $\bm{Q} = \bm{Q}^{'}\bm{R}$ so that $\bm{Q}_1 = \bm{Q}_1^{'} \bm{R}.$ Because the matrix $\bm{Q}_1$ is nonsingular, its left polar decomposition into the product of a symmetric positive definite matrix and an orthogonal matrix is unique (see, for example, Proposition 5.5 of \citet{Eaton1983}). We conclude that $\bm{R} = \bm{I}_k$ and $\bm{Q}_1 = \bm{Q}_1^{'}.$ Thus, $l$ is injective on $\mathcal{V}^{+}(k,p).$

Next, we prove that the image of $\mathcal{V}^{+}(k,p)$ under $l$ has measure one with respect to the uniform probability measure $P_{\mathcal{G}(k,p)}$ on $\mathcal{G}(k,p).$ Define $\mathcal{V}^{\text{N}}(k,p)$ as the set
\begin{align*}
    \mathcal{V}^{\text{N}}(k,p) &= \left\{\bm{Q}=\left[\bm{Q}_1^T \quad \bm{Q}_2^T\right]^T \in \mathcal{V}(k,p) : \bm{Q}_1 \text{ is nonsingular}\right\}.
\end{align*} The set $\mathcal{V}^{\text{N}}(k,p)$ has measure one with respect to $P_{\mathcal{V}(k,p)}$ and the following lemma holds:  
\begin{lemma}
The images of $\mathcal{V}^{+}(k,p)$ and $\mathcal{V}^{\text{N}}(k,p)$ under $l$ are equal. 
\end{lemma}

\begin{proof}
The direction $l\left[\mathcal{V}^{+}(k,p)\right] \subseteq l\left[\mathcal{V}^{\text{N}}(k,p)\right]$ is immediate. Now, let $S \in l\left[\mathcal{V}^{\text{N}}(k,p)\right].$ There must exist $\bm{Q} = \left[\bm{Q}_1^T \, \bm{Q}_2^T\right]^T \in \mathcal{V}^{\text{N}}(k,p)$ having $S$ as its column space. Let $\bm{Q}_1 = \bm{U}\bm{D}\bm{V}^T$ be the singular value decomposition of $\bm{Q}_1$ and set $\bm{Q}' = \bm{Q}\bm{V}\bm{U}^T.$ Then  $l\left(\bm{Q}'\right) = S$ because $\bm{V}\bm{U}^T \in \mathcal{O}(k)$ and $\bm{Q}'\in \mathcal{V}^{+}(k,p)$ because its square top block $\bm{U}\bm{D}\bm{U}^T$ is symmetric positive definite. Thus $S \in l\left[\mathcal{V}^{+}(k,p)\right]$ and we conclude that $l\left[\mathcal{V}^{\text{N}}(k,p)\right] \subseteq l\left[\mathcal{V}^{+}(k,p)\right].$
\end{proof} \noindent Recall that the measure $P_{\mathcal{G}(k,p)}$ is the pushforward of $P_{\mathcal{V}(k,p)}$ by $l,$ i.e. for a subset $A \subset \mathcal{G}(k,p)$ we have $P_{\mathcal{G}(k,p)}(A) = P_{\mathcal{V}(k,p)}\left[l^{-1}(A)\right].$ Thus,
\begin{align*}
    P_{\mathcal{G}(k,p)}\left\{l\left[\mathcal{V}^{+}(k,p)\right]\right\} &= P_{\mathcal{G}(k,p)}\left\{l\left[\mathcal{V}^{\text{N}}(k,p)\right]\right\} \\ 
    &= P_{\mathcal{V}(k,p)}\left\{\mathcal{V}^{\text{N}}(k,p)\right\} \\ 
    &=1. 
\end{align*}

\subsection{Proof of Proposition \ref{cayley_1to1}}

We begin with $\bm{Q} =[\bm{Q}_1^T \, \bm{Q}_2^T] \in \mathcal{V}^{+}(k,p)$ and we want to verify that the matrix $\bm{A}$ obtained by equations \ref{invgrassF}-\ref{invgrassA} satisfies $\text{eval}_i(\bm{A}^T\bm{A}) \in [0,1)$ for $1 \leq i \leq k.$  Let $\bm{Q}_1 = \bm{V}\text{diag}(\lambda_1, ..., \lambda_k)\bm{V}^T$ be the eigendecomposition of $\bm{Q}_1.$ That each eigenvalue is positive follows from the condition $\bm{Q}_1 \succ \bm{0}.$ We know that each eigenvalue is less than or equal to one because 
\begin{align*}
    \bm{0}_k &\preceq \bm{Q}_2^T\bm{Q}_2 \\ 
    &= \bm{I}_k - \bm{Q}_1^T\bm{Q}_1\\
    &= \bm{V}\text{diag}\left(1-\lambda_1^2, ..., 1-\lambda_k^2\right)\bm{V}^T. 
\end{align*} Therefore, $\lambda_i \in (0,1]$ for each $i.$ As in equations \ref{invgrassF}-\ref{invgrassA}, set
\begin{align*}
    \bm{F} &= (\bm{I}_k-\bm{Q}_1)(\bm{I}_k+\bm{Q}_1)^{-1} \\ 
    \bm{A} &= \frac12 \bm{Q}_2 (\bm{I}_k + \bm{F}).
\end{align*} We can write $\bm{A}^T\bm{A}$ as 
\begin{align*}
    \bm{A}^T\bm{A} &= \frac14 (\bm{I}_k + \bm{F})^T \bm{Q}_2^T \bm{Q}_2(\bm{I}_k + \bm{F}) \\ 
    &= \frac14 (\bm{I}_k + \bm{F})^T \left(\bm{I}_k - \bm{Q}_1^T \bm{Q}_1\right)(\bm{I}_k + \bm{F})
\end{align*}
from which it follows that 
\begin{align*}
    \text{eval}_i(\bm{A}^T\bm{A}) &= \frac14\left(1-\lambda_i^2\right)\left(1 + \frac{1-\lambda_i}{1+\lambda_i}\right)^2
\end{align*} for each $i.$ Since the eigenvalues of $\bm{Q}_1$ lie in the interval $(0,1],$ the eigenvalues $\bm{A}^T\bm{A}$ lie in the interval $[0,1).$ 

Now we start with $\bm{A}$ such that $\delta_i = \text{eval}_i(\bm{A}^T\bm{A}) \in [0,1)$ for each $i$ and we want to check that $\bm{Q}_1 = (\bm{I}_k - \bm{A}^T\bm{A})(\bm{I}_k + \bm{A}^T\bm{A})^{-1} \succ \bm{0}.$ Let $\bm{A}^T\bm{A} = \bm{W} \text{diag}(\delta_1, ..., \delta_k)\bm{W}^T$ be the eigendecomposition of $\bm{A}^T\bm{A}.$ Then 
\begin{align*}
    \bm{Q}_1 &= \bm{W} \text{diag}\left(\frac{1-\delta_1}{1+\delta_1}, ..., \frac{1-\delta_k}{1+\delta_k}\right)\bm{W}^T \succ \bm{0}.
\end{align*} 

\subsection{Proof of Proposition \ref{jacobian_grassmann}}

The proof is nearly identical to that of Proposition \ref{jacobian_stiefel}. We simply replace $\bm{\varphi}$ with $\bm{\psi}$ and $\bm{\Gamma}_{\mathcal{V}}$ with $\bm{\Gamma}_{\mathcal{G}}.$ 

\subsection{Proof of Proposition \ref{normal_apprx_prop} part (i)} 

Denote the square top block of $\bm{Q}_p$ by $\bm{Q}_{p,1}$ and the bottom block by $\bm{Q}_{p,2}.$ Define matrices
\begin{align*} 
    \bm{F}_p &= (\bm{I}_{k_p} - \bm{Q}_{p,1})(\bm{I}_{k_p} + \bm{Q}_{p,1})^{-1} \\ 
    \bm{B}_p &= \frac12 (\bm{F}_p^T - \bm{F}_p) \\ 
    \bm{A}_p &= \frac12 \bm{Q}_{p,2} (\bm{I}_{k_p} + \bm{F}_p) 
\end{align*} as in equations \ref{inverse_eqn_F}-\ref{inverse_eqn_A}. Let $\bm{b}_p$ be the vector of independent entries of $\bm{B}_p$ obtained by eliminating diagonal and supradiagonal elements from $\vect{\bm{B}_p}.$  When the Frobenius norm $\|\bm{Q}_{p,1}\|_{F}$ is less than one, the matrices admit series representations: 
 \begin{align*}
     \bm{F}_p &= \bm{I}_{k_p}  - 2\bm{Q}_{p,1} + 2\bm{Q}_{p,1}^2 - 2\bm{Q}_{p,1}^3 + ... \\
     \bm{B}_p &= (\bm{Q}_{p,1} - \bm{Q}_{p,1}^T) - (\bm{Q}_{p,1}^2 - \bm{Q}_{p,1}^{2T}) + (\bm{Q}_{p,1}^3 - \bm{Q}_{p,1}^{3T}) - ... \\ 
     \bm{A}_p &= \bm{Q}_{p,2} - \bm{Q}_{p,2} \bm{Q}_{p,1} + \bm{Q}_{p,2} \bm{Q}_{p,1}^2 - \bm{Q}_2 \bm{Q}_{p,1}^3 + ... 
 \end{align*} It follows that 
\begin{align*}
     \widetilde{\bm{B}}_p(\bm{Q}_p) - \bm{B}_p &= \sum_{i=2}^{\infty} (-1)^{i} \left(\bm{Q}^{i}_{p,1} - \bm{Q}^{iT}_{p,1}\right) \\
     \widetilde{\bm{A}}_p(\bm{Q}_p) - \bm{A}_p &=  \sum_{i=1}^{\infty} (-1)^{i-1}\bm{Q}_{p,2}\bm{Q}^{i}_{p,1}
\end{align*} when $\|\bm{Q}_{p,1}\|_{F} < 1.$ See \citet{Hubbard2009} for a discussion of matrix geometric series.

Inequalities relating the Frobenius and max norms will prove useful. Let $\bm{W}_1, \bm{W}_2$ be $d_1 \times d_2$ and $d_2 \times d_3$ dimensional matrices, respectively. Then
\begin{align*}
    \|\bm{W}_1\|_{F} &\leq \sqrt{d_1 d_2} \|\bm{W}_1\|_{\max} \\ 
    \|\bm{W}_1\bm{W}_2\|_{\max} &\leq d_2 \|\bm{W}_1\|_{\max} \|\bm{W}_2\|_{\max}.
\end{align*}
The first inequality implies that the condition $\|\bm{Q}_{p,1}\|_F < 1$ under which our series representations converge is satisfied when $k_p\|\bm{Q}_{p,1}\|_{\max} < 1.$ The second inequality implies that ${\|\bm{Q}_{p,1}^i\|_{\max} \leq k_p^{i-1} \|\bm{Q}_{p,1}\|_{\max}^i}$ for each natural number $i.$

We will also need the following somewhat technical lemma: 
\begin{lemma} \label{lemma} The following quantities, which will appear in later inequalities, go to zero in probability as $p$ grows provided that ${k_p=o\left(\frac{p^{1/4}}{\sqrt{\log{p}}}\right):}$
\begin{align*}
    &(i)\quad \|p^{1/2}\bm{Q}_p - \bm{Z}_p \|_{\max}  &&(ii)\quad  k_p\|\bm{Q}_{p,1}\|_{\max}  \\ 
    &(iii)\quad  k_p^2 p^{1/2}\|\bm{Q}_{p,1}\|_{\max}^2  &&(iv)\quad  k_p p^{1/2}\|\bm{Q}_{p}\|_{\max}^2.  \\ 
\end{align*}
\end{lemma} 
\begin{proof}
Suppose ${k_p=o\left(\frac{p^{1/4}}{\sqrt{\log{p}}}\right).}$ This implies that ${k_p=o\left(\frac{p}{\log{p}}\right)}$ and quantity (i) goes to zero in probability by Theorem 3 of \citet{Jiang2006}. The quantities (ii)-(iv) are nonnegative and each is smaller than either $k_p^2\, p^{1/2}\|\bm{Q}_{p}\|^2_{\max}$ or its square root. If we can show that either this quantity or its square root goes to zero in probability, we are done. We have the following upper bound for the square root:
\begin{align*}
    k_p\, p^{1/4} \left\|\bm{Q}_p\right\|_{\max} &= k_p \,p^{1/4}\left\|\bm{Q}_p -p^{-1/2}\bm{Z}_p + p^{-1/2}\bm{Z}_p\right\|_{\max} \\ 
    &\leq k_p \,p^{1/4} \left(\left\|\bm{Q}_p -p^{-1/2}\bm{Z}_p\right\|_{\max} + \left\|p^{-1/2}\bm{Z}_p\right\|_{\max}\right) \\
    &= k_p \,p^{-1/4} \left\|p^{1/2}\bm{Q}_p -\bm{Z}_p\right\|_{\max}  + k_p\, p^{-1/4}\left\|\bm{Z}_p\right\|_{\max}.
\end{align*} We know that the first summand goes to zero in probability by the condition on $k_p$ and the fact that quantity (i) goes to zero in probability. Using a well-known inequality involving the maximum of independent standard normal random variables (see Problem 5.1 in \citet{VanHandel2014}), we obtain the following bound on the expected value of the second summand: 
\begin{align*}
    k_p\, p^{-1/4}\mathbbm{E}\left[\left\|\bm{Z}_p\right\|_{\max}\right] 
    &\leq k_p\, p^{-1/4} \sqrt{2 \log{pk_p}} \\ 
    &\leq k_p\, p^{-1/4} \sqrt{2 \log{p^2}} \\ 
    &=2 \frac{k_p}{\left(\frac{p^{1/4}}{\sqrt{\log{p}}}\right)}.
\end{align*} The condition on $k_p$ implies that the expectation of the second summand goes to zero. We conclude that the second summand goes to zero in mean and thus in probability. We have shown that an upper bound for the square root of $k_p^2\, p^{1/2}\|\bm{Q}_{p}\|_{\max}^2$ goes to zero in probability, which is sufficient to prove the lemma. 
\end{proof}

Now set $a_p = \|\bm{\Pi}_p\widetilde{\bm{\varphi}}_p - \bm{\Pi}_p\bm{\varphi}_p\|_{\infty},$ assume that ${k_p=o\left(\frac{p^{1/4}}{\sqrt{\log{p}}}\right)},$ and let $\epsilon > 0$ be given.  We want to show that $\text{Pr}\{a_p > \epsilon \} \to 0$ as $p$ gets large. We express this probability as 
\begin{align*}
    \text{Pr}\{a_p > \epsilon \} = \, &\text{Pr}\left\{a_p > \epsilon \, \mid \, k_p\|\bm{Q}_{p,1}\|_{\max} < 1\right\}\text{Pr}\left\{k_p\|\bm{Q}_{p,1}\|_{\max} < 1\right\} + \\ 
    & \text{Pr}\left\{a_p > \epsilon \, \mid \, k_p\|\bm{Q}_{p,1}\|_{\max} \geq 1\right\}\text{Pr}\left\{k_p\|\bm{Q}_{p,1}\|_{\max} \geq 1\right\} \\ 
    \leq \, &\text{Pr}\left\{a_p > \epsilon \, \mid \, k_p\|\bm{Q}_{p,1}\|_{\max} < 1\right\} + 
    \text{Pr}\left\{k_p\|\bm{Q}_{p,1}\|_{\max} \geq 1\right\}.
\end{align*} Lemma \ref{lemma} implies that $\text{Pr}\{ k_p\|\bm{Q}_{p,1}\|_{\max} \geq 1\}$ goes to zero. It follows that ${\text{Pr}\{a_p > \epsilon \}}$ goes to zero if ${\text{Pr}\left\{a_p > \epsilon \, \mid \, k_p\|\bm{Q}_{p,1}\|_{\max} < 1\right\}}$ does. Therefore, we only need to verify this condition. 

For each $p,$ we have
\begin{align*}
    a_p &= \|\bm{\Pi}_p\widetilde{\bm{\varphi}}_p - \bm{\Pi}_p\bm{\varphi}_p\|_{\infty} \\ 
    &= 
    \left\|
    \begin{bmatrix}
           \sqrt{p/2}\, \widetilde{\bm{b}}_p(\bm{Q}_p) \\ 
           \sqrt{p} \vect{\widetilde{\bm{A}}_p}(\bm{Q}_p)
    \end{bmatrix} - 
    \begin{bmatrix}
           \sqrt{p/2}\, \bm{b}_p \\ 
           \sqrt{p} \vect{\bm{A}_p}                    
    \end{bmatrix}
    \right\|_{\infty} \\ 
    &= \max\left\{\sqrt{p/2}\, \left\|\widetilde{\bm{B}}_p(\bm{Q}_p) - \bm{B}_p\right\|_{\max} \, , \, 
    \sqrt{p}\, \left\|\widetilde{\bm{A}}_p(\bm{Q}_p) - \bm{A}_p\right\|_{\max}
    \right\} \\ 
    &\leq \sqrt{p/2}\, \left\|\widetilde{\bm{B}}_p(\bm{Q}_p) - \bm{B}_p\right\|_{\max}  \, + \, \sqrt{p} \left\|\widetilde{\bm{A}}_p(\bm{Q}_p) - \bm{A}_p\right\|_{\max}.
\end{align*} When $k_p\|\bm{Q}_{p,1}\|_{\max} < 1,$ we have
\begin{align*}
    \sqrt{p/2}\left\|\widetilde{\bm{B}}_p(\bm{Q}_p) - \bm{B}_p\right\|_{\max} &= \sqrt{p/2}\left\|\sum_{i=2}^{\infty} (-1)^i \left(\bm{Q}^{i}_{p,1} - \bm{Q}^{iT}_{p,1}\right)\right\|_{\max} \\ 
    &\leq\sqrt{p/2}\sum_{i=2}^{\infty}  \left\|\bm{Q}^{i}_{p,1} - \bm{Q}^{iT}_{p,1}\right\|_{\max} \\
    &\leq \sqrt{p/2} \sum_{i=2}^{\infty} 2  \left\|\bm{Q}^{i}_{p,1}\right\|_{\max} \\
    &= \sqrt{2p} \sum_{i=2}^{\infty}  \left\|\bm{Q}^{i}_{p,1}\right\|_{\max} \\
    &\leq\sqrt{2p} \sum_{i=2}^{\infty}  k_p^{i-1}\left\|\bm{Q}_{p,1}\right\|_{\max}^i \\
    &= \sqrt{2p} \sum_{n=i}^{\infty}  \left(k_p^\frac{i-1}{i}\left\|\bm{Q}_{p,1}\right\|_{\max}\right)^i \\
    &\leq \sqrt{2p} \sum_{i=2}^{\infty}  \left(k_p\left\|\bm{Q}_{p,1}\right\|_{\max}\right)^i \\
    &= \sqrt{2p}\left(\frac{1}{1 - k_p\left\|\bm{Q}_{p,1}\right\|_{\max}} -  k_p\left\|\bm{Q}_{p,1}\right\|_{\max} - 1 \right) \\ 
    &= \sqrt{2}\frac{k_p^2 \sqrt{p}\left\|\bm{Q}_{p,1}\right\|_{\max}^2 }{1 - k_p\left\|\bm{Q}_{p,1}\right\|_{\max}} 
\end{align*} and  
\begin{align*}
    \sqrt{p}\left\|\widetilde{\bm{A}}_p(\bm{Q}_p) - \bm{A}_p\right\|_{\max} &= \sqrt{p}\left\|\sum_{i=1}^{\infty} (-1)^{i-1}\bm{Q}_{p,2}\bm{Q}^{i}_{p,1}\right\|_{\max} \\
    &\leq \sqrt{p} \sum_{n=1}^{\infty} \left\|\bm{Q}_{p,2}\bm{Q}^n_{p,1} \right\|_{\max} \\
    &\leq \sqrt{p} \sum_{n=1}^{\infty} k_p \left\|\bm{Q}_{p,2} \right\|_{\max} \left\|\bm{Q}^n_{p,1} \right\|_{\max} \\
    &\leq \sqrt{p} \left\|\bm{Q}_{p,2} \right\|_{\max} \sum_{n=1}^{\infty} \left(k_p \left\|\bm{Q}_{p,1} \right\|_{\max}\right)^n \\
    &= \sqrt{p}\left\|\bm{Q}_{p,2}\right\|_{\max} \left( \frac{1}{1 - k_p\left\| \bm{Q}_{p,1} \right\|_{\max}} - 1 \right)\\
    &= \sqrt{p}\left\|\bm{Q}_{p,2}\right\|_{\max} \left( \frac{k_p\left\| \bm{Q}_{p,1} \right\|_{\max}}{1 - k_p\left\| \bm{Q}_{p,1} \right\|_{\max}} \right) \\ 
    &\leq \frac{k_p \sqrt{p} \max\left\{\left\| \bm{Q}_{p,1} \right\|_{\max}, \left\|\bm{Q}_{p,2}\right\|_{\max}\right\}^2} {1 - k_p\left\| \bm{Q}_{p,1} \right\|_{\max}} \\ 
    &= \frac{k_p \sqrt{p} \left\| \bm{Q}_{p} \right\|_{\max}^2} {1 - k_p\left\| \bm{Q}_{p,1} \right\|_{\max}}.
\end{align*} Thus, the upper bound
\begin{align*}
    a_p &\leq \sqrt{p/2}\, \left\|\widetilde{\bm{B}}_p(\bm{Q}_p) - \bm{B}_p\right\|_{\max}  \, + \, \sqrt{p} \left\|\widetilde{\bm{A}}_p(\bm{Q}_p) - \bm{A}_p\right\|_{\max} \\ 
    &\leq \sqrt{2}\frac{k_p^2 \sqrt{p}\left\|\bm{Q}_{p,1}\right\|_{\max}^2 }{1 - k_p\left\|\bm{Q}_{p,1}\right\|_{\max}}  + \frac{k_p \sqrt{p} \left\| \bm{Q}_{p} \right\|_{\max}^2} {1 - k_p\left\| \bm{Q}_{p,1} \right\|_{\max}} := u_p
\end{align*} is valid when $k_p\|\bm{Q}_{p,1}\|_{\max} < 1.$ Then
\begin{align*}
    \Pr\left\{a_p > \epsilon\, \mid \,  k_p\|\bm{Q}_{p,1}\|_{\max} < 1 \right\} &\leq \Pr\left\{u_p > \epsilon\, \mid\,  k_p\|\bm{Q}_{p,1}\|_{\max} < 1\right\} \\ 
    &= \frac{\Pr\left\{u_p > \epsilon\, \cap \, k_p\|\bm{Q}_{p,1}\|_{\max} < 1 \right\}}{\Pr\left\{k_p\|\bm{Q}_{p,1}\|_{\max} < 1\right\}} \\
    &\leq 
     \frac{\Pr\left\{ u_p > \epsilon \right\} }{\Pr\left\{k_p\|\bm{Q}_{p,1}\|_{\max} < 1\right\}}.
\end{align*} Because $\Pr\left\{ k_p\|\bm{Q}_{p,1}\|_{\max} < 1\right\} \to 1,$ we only need to show $\Pr\left\{u_p > \epsilon \right\} \to 0$ as $p$ grows. Since $\epsilon$ is arbitrary, this is equivalent to showing that $u_p$ goes to zero in probability, which follows from the continuous mapping theorem and Lemma \ref{lemma}. We have shown that $ \Pr\left\{a_p > \epsilon\, \mid \,  k_p\|\bm{Q}_{p,1}\|_{\max} < 1 \right\}$ goes to zero as $p$ gets large, which is sufficient to prove part (i) of the proposition. 

\subsection{Proof of Proposition \ref{normal_apprx_prop} part (ii)}
Assume that $k_p  = o\left(\frac{p}{\log{p}}\right).$ For each $p,$ we have: 
\begin{align*}
    \|\bm{\Pi}_p\widetilde{\bm{\varphi}}_p - \bm{z}_p\|_{\infty} &= \|\bm{\Pi}_p \widetilde{C}_p^{-1}(\bm{Q}_{p}) - \bm{\Pi}_p \,  \widetilde{C}_p^{-1}(p^{-1/2}\bm{Z}_{p})\|_{\infty} \\
    &= 
    \left\|
    \begin{bmatrix}
    \sqrt{p/2} \widetilde{b}_p(\bm{Q}_p) \\  
    \sqrt{p} \vect{\widetilde{A}_p(\bm{Q}_p)}
    \end{bmatrix} - 
    \begin{bmatrix}
     \sqrt{p/2} \widetilde{b}_p(p^{-1/2}\bm{Z}_p) \\  
     \sqrt{p} \vect{\widetilde{A}_p}(p^{-1/2}\bm{Z}_p)                   
    \end{bmatrix}
    \right\|_{\infty} \\
    &= \max\left\{
    \sqrt{p/2}
    \left\| 
     \widetilde{B}_p(\bm{Q}_p) - \widetilde{B}_p(p^{-1/2}\bm{Z}_p) 
    \right\|_{\max} , \right. \\ 
     &\left. \hspace{1.5cm} \sqrt{p}
    \left\| 
     \widetilde{A}_p(\bm{Q}_p) - \widetilde{A}_p(p^{-1/2}\bm{Z}_p)
    \right\|_{\max}
    \right\} \\
    &\leq 
    \sqrt{p/2}
    \left\| 
     \widetilde{B}_p(\bm{Q}_p) - \widetilde{B}_p(p^{-1/2}\bm{Z}_p) 
    \right\|_{\max} + \\ 
    &\quad \sqrt{p}
    \left\| 
     \widetilde{A}_p(\bm{Q}_p) - \widetilde{A}_p(p^{-1/2}\bm{Z}_p)
    \right\|_{\max} \\ 
    &=
    \sqrt{p/2}
    \left\| 
    \bm{Q}_{p,1} - p^{-1/2}\bm{Z}_{p,1} + p^{-1/2}\bm{Z}_{p,1}^T - \bm{Q}_{p,1}^T
    \right\|_{\max} + \\ 
    & \quad \,\, \sqrt{p}
    \left\| 
     \bm{Q}_{p,2} - p^{-1/2}\bm{Z}_{p,2}
    \right\|_{\max} \\ 
    &\leq 
    \sqrt{p/2} \left(
    \left\| \bm{Q}_{p,1} - p^{-1/2}\bm{Z}_{p,1} \right\| + \left\|p^{-1/2}\bm{Z}_{p,1}^T - \bm{Q}_{p,1}^T \right\|_{\max} 
    \right) + \\ 
    & \quad \,\, \sqrt{p}
    \left\| 
     \bm{Q}_{p,2} - p^{-1/2}\bm{Z}_{p,2}
    \right\|_{\max} \\ 
    &= 2\sqrt{p/2}
    \left\| \bm{Q}_{p,1} - p^{-1/2}\bm{Z}_{p,1} \right\|_{\max} + 
    \sqrt{p} \left\| \bm{Q}_{p,2} - p^{-1/2}\bm{Z}_{p,2} \right\|_{\max} \\ 
    &= \sqrt{2} \left\| \sqrt{p}\bm{Q}_{p,1} - \bm{Z}_{p,1} \right\|_{\max} + 
    \left\| \sqrt{p}\bm{Q}_{p,2} - \bm{Z}_{p,2} \right\|_{\max} \\
    &\leq \sqrt{2} \left\| \sqrt{p}\bm{Q}_{p} - \bm{Z}_{p} \right\|_{\max} + 
    \left\| \sqrt{p}\bm{Q}_{p} - \bm{Z}_{p} \right\|_{\max} \\
    &= (\sqrt{2} + 1) \left\| \sqrt{p}\bm{Q}_{p} - \bm{Z}_{p} \right\|_{\max}.
\end{align*} Theorem 3 of \citet{Jiang2006} implies that this upper bound goes to zero in probability as $p$ grows. Therefore, the quantity $\|\bm{\Pi}_p\widetilde{\bm{\varphi}}_p - \bm{z}_p\|_{\infty}$ does as well. 

\section{Special matrices} \label{special_matrices}

In constructing the linear transformations given in Lemmas \ref{lintran_stiefel} and \ref{lintran_grassmann}, we rely upon two special matrices: the commutation matrix $\bm{K}_{m,n}$ and the matrix $\widetilde{\bm{D}}_n.$ An early reference related to the matrix $\bm{K}_{m,n}$ is \citet{Magnus1979}, while the matrix $\widetilde{\bm{D}}_n$ was introduced in \citet{Neudecker1983}. Our presentation follows that of \citet{Magnus1988LS}. 

\subsection{The commutation matrix $\bm{K}_{m,n}$}
Let $\bm{A}$ be an $m\times n$ matrix and $B$ be a $p\times q$ matrix. The commutation matrix $\bm{K}_{m,n}$ is the unique $mn \times mn$ permutation matrix with the property that $\bm{K}_{m,n} \vect{A} = \vect{A^T}.$ The critical property of the commutation matrix is that it allows us to exchange the order of the matrices in a Kronecker product.  Theorem 3.1 of \citet{Magnus1988LS} states that $\bm{K}_{p,m}(\bm{A} \otimes \bm{B}) = (\bm{B} \otimes \bm{A})\bm{K}_{q,n}.$ Section 3.3 of \citet{Magnus1988LS} gives an explicit expression for the commutation matrix $\bm{K}_{m,n}.$ Let $\bm{H}_{i,j}$ be the $m\times n$ matrix having a $1$ in the $i,j$th position and zeros everywhere else. Then Theorem 3.2 of \citet{Magnus1988LS} tells us that
\begin{align*}
    \bm{K}_{m,n} &= \sum_{i=1}^{m} \sum_{j=1}^n \left(\bm{H}_{i,j} \otimes \bm{H}_{i,j}^T\right).
\end{align*}

\subsection{The matrix $\widetilde{\bm{D}}_n$}

The matrix $\widetilde{\bm{D}}_n$ proves useful when differentiating expressions involving skew-symmetric matrices. Let $\bm{A}$ be an $n\times n$ matrix and let $\tilde{v}(A)$ be the $n(n-1)/2$-dimensional vector obtained by eliminating diagonal and supradiagonal elements from the vectorization $\vect{\bm{A}}.$ Definition 6.1 of \citet{Magnus1988LS} defines $\widetilde{\bm{D}}_n$ as the unique $n^2 \times n(n-1)/2$ matrix with the property $\widetilde{\bm{D}}_n \tilde{v}(A) = \vect{\bm{A}}$ for every skew-symmetric matrix $\bm{A}.$ Theorem 6.1 of \citet{Magnus1988LS} gives an explicit expression for $\widetilde{\bm{D}}_n.$ Let $\bm{E}_{i,j}$ be the $n\times n$ matrix with $1$ in the $i,j$th position and zeros everywhere else and set $\tilde{\bm{T}}_{i,j} = \bm{E}_{i,j} - \bm{E}_{j,i}.$ Also, let $\tilde{u}_{i,j}$ be the $n(n-1)/2$-dimensional vector having $1$ in its $(j-1)n + i - j(j+1)/2$ place and zeros everywhere else. Then Theorem 6.1 tells us that 
\begin{align*}
    \widetilde{\bm{D}}_n &= \sum_{i > j} \left(\vect{\tilde{\bm{T}}_{i,j}}\right) \tilde{u}_{i,j}^T. 
\end{align*}

\section{Evaluating the Jacobian terms}

Taking a naive approach to evaluating the Jacobian term $J_{d_{\mathcal{V}}}C(\bm{\varphi})$ becomes prohibitively expensive for even small dimensions. Recall that 
\begin{align*}
    J_{d_{\mathcal{V}}}C(\bm{\varphi}) &= \left|DC(\bm{\varphi})^T DC(\bm{\varphi})\right|^{1/2} \\ 
    &= \left|2^2\, \bm{\Gamma}_{\mathcal{V}}^T \,(\bm{G}_{\mathcal{V}} \otimes \bm{H}_{\mathcal{V}})\, \bm{\Gamma}_{\mathcal{V}}\right|^{1/2}
\end{align*} where
\begin{align*}
    \bm{G}_{\mathcal{V}} &= (\bm{I}_p - \bm{X}_{\bm{\varphi}})^{-1}\bm{I}_{p\times k}\bm{I}_{p\times k}^T (\bm{I}_p - \bm{X}_{\bm{\varphi}})^{-T}\\ 
    \bm{H}_{\mathcal{V}} &= (\bm{I}_p - \bm{X}_{\bm{\varphi}})^{-T}(\bm{I}_p - \bm{X}_{\bm{\varphi}})^{-1}.
\end{align*} The Kronecker product $\bm{G}_{\mathcal{V}} \otimes \bm{H}_{\mathcal{V}}$ has dimension $p^2 \times p^2.$ Evaluating this Kronecker product and computing its matrix product with $\bm{\Gamma}_{\mathcal{V}}$ is extremely costly for large $p.$ In this section, we describe a more efficient approach which takes advantage of the block stucture of the matrices involved. (The Jacobian term $J_{d_{\mathcal{G}}}C(\bm{\psi})$ can be evaluated analogously.) 

Let $\bm{C}_{\mathcal{V}} = \left(\bm{I} - \bm{X}_{\bm{\varphi}}\right)^{-1}.$ Then 
\begin{align*}
    \bm{C}_{\mathcal{V}} &= 
    \begin{bmatrix}
                    \bm{C}_{11} &  \bm{C}_{12} \\ 
                    \bm{C}_{21} &\bm{C}_{22}
    \end{bmatrix} 
\end{align*} where 
\begin{align*}
    \bm{C}_{11} &= (\bm{I}_k - \bm{B} + \bm{A}^T\bm{A})^{-1} 
    &&\bm{C}_{12} = -\bm{C}_{11}\bm{A}^T \\ 
    \bm{C}_{21} &= \bm{A}\bm{C}_{11}  
    &&\bm{C}_{22} = \bm{I}_{p-k} - \bm{A}\bm{C}_{11}\bm{A}^T. \\ 
\end{align*} The blocks of the matrices 
\begin{align*}
    \bm{G}_{\mathcal{V}} = \begin{bmatrix}
    \bm{G}_{11} & \bm{G}_{12} \\ 
    \bm{G}_{21} & \bm{G}_{22}
    \end{bmatrix} \quad
    \bm{H}_{\mathcal{V}} = \begin{bmatrix}
    \bm{H}_{11} & \bm{H}_{12} \\ 
    \bm{H}_{21} & \bm{H}_{22}
    \end{bmatrix}
\end{align*} can be written in terms of the blocks of $\bm{C}_{\mathcal{V}}$ as 
\begin{align*}
    \bm{H}_{11} &= \bm{C}_{11}^T\bm{C}_{11} + \bm{C}_{21}^T \bm{C}_{21}  
    &&\bm{H}_{12} = \bm{C}_{11}^T \bm{C}_{12} + \bm{C}_{21}^T \bm{C}_{22} \\ 
    \bm{H}_{21} &= \bm{C}_{12}^T\bm{C}_{11} + \bm{C}_{22}^T \bm{C}_{21}  
    &&\bm{H}_{22} = \bm{C}_{12}^T \bm{C}_{12} + \bm{C}_{22}^T\bm{C}_{22} \\
\end{align*} and 
\begin{align*}
    \bm{G}_{11} &= \bm{C}_{11}\bm{C}_{11}^T  
    &&\bm{G}_{12} = \bm{C}_{11} \bm{C}_{21}^T \\  
    \bm{G}_{21} &= \bm{C}_{21}\bm{C}_{11}^T 
    &&\bm{G}_{22} = \bm{C}_{21} \bm{C}_{21}^T. \\  
\end{align*} We can express the matrix $DC(\bm{\varphi})^T DC(\bm{\varphi})$ in blocks as
\begin{align*}
     DC(\bm{\varphi})^T DC(\bm{\varphi}) &= 2^2 \bm{\Gamma}_{\mathcal{V}}^T\, (\bm{G}_{\mathcal{V}} \otimes \bm{H}_{\mathcal{V}}) \,\bm{\Gamma}_{\mathcal{V}} \\ 
     &= 2^2\begin{bmatrix}
     \bm{\Omega}_{11} & \bm{\Omega}_{12} \\ 
     \bm{\Omega}_{21} & \bm{\Omega}_{22}
     \end{bmatrix}
\end{align*} where 
\begin{align*}
    \bm{\Omega}_{11} &= \widetilde{\bm{D}}_k^T(\bm{G}_{11} \otimes \bm{H}_{11})\widetilde{\bm{D}}_k \\ 
    \bm{\Omega}_{12} &= \widetilde{\bm{D}}_k^T(\bm{G}_{11} \otimes \bm{H}_{12}) - \widetilde{\bm{D}}_k^T(\bm{G}_{12} \otimes \bm{H}_{11})\bm{K}_{p-k,k} \\ 
    \bm{\Omega}_{21} &= \bm{\Omega}_{12}^T \\ 
    \bm{\Omega}_{22} &= (\bm{G}_{11} \otimes \bm{H}_{22} + \bm{H}_{11} \otimes \bm{G}_{22}) - (\bm{G}_{12} \otimes \bm{H}_{21} + \bm{H}_{12} \otimes \bm{G}_{21})\bm{K}_{p-k,k}.
\end{align*} Then
\begin{align*}
    J_{d_{\mathcal{V}}}C(\bm{\varphi})
    &= \left|2^{2}
    \begin{bmatrix}
     \bm{\Omega}_{11} & \bm{\Omega}_{12} \\ 
     \bm{\Omega}_{21} & \bm{\Omega}_{22}
     \end{bmatrix}
    \right|^{1/2} \\ 
    &= 2^{d_{\mathcal{V}}} \left|\bm{\Omega}_{22}\right| \left|\bm{\Omega}_{11} - \bm{\Omega}_{12}\bm{\Omega}_{22}^{-1}\bm{\Omega}_{21}\right|.
\end{align*} The Kronecker products in the formulas for $\bm{\Omega}_{11}, \bm{\Omega}_{12}, \bm{\Omega}_{21},$ and $\bm{\Omega}_{22}$ are much smaller than $\bm{G}_{\mathcal{V}} \otimes \bm{H}_{\mathcal{V}},$ which leads to a significant computational savings compared to the naive approach. 

\end{document}